\documentclass[a4paper]{article}
\usepackage[ansinew]{inputenc}
\usepackage[english]{babel}
\usepackage{amsbsy,amscd,amsfonts,amsmath,amsopn,amssymb,amstext,amsthm,amsxtra,array,color}
\usepackage{fancybox,float,graphicx,latexsym,subfigure,srcltx,times,tikz,url,amsthm}
\usepackage{enumerate}
\newcommand{\diam}{\operatorname{diam}} 
\newcommand{\Or}{\operatorname{Or}} 

\newcommand{\ro}[2]{\row(#1, #2)}
\newcommand{\nur}[2]{\nu(\ro{#1}{#2})}

\newcommand{\norma}[1]{M_{#1}^{N}}
\newcommand{\snorma}[1]{M_{#1}^{SN}}

\newcommand{\mx}[1]{W_n}

\newcommand{\sucen}[2]{{#1_{1},\hdots,#1_{#2}}}

\newcommand{\N}{\mathbb{N}}

\newcommand{\R}{\mathbb{R}}

\newcommand{\MM}{\mathfrak{M}}

\newcommand{\ORTHO}{\mathcal{ORTHO}}
\newcommand{\VNL}{\mathcal{VNL}}
\newcommand{\WNL}{\mathcal{WNL}}

\newcommand{\gift}{\operatorname{gift}}
\newcommand{\propa}{\operatorname{prop}}

\renewcommand{\int}{\operatorname{int}}

\newcommand{\row}{\operatorname{row}}

\newcommand{\dist}{\operatorname{dist}}

\newcommand{\girth}{\operatorname{girth}}

\newtheorem{thm}{Theorem}[section]
\newtheorem{lem}[thm]{Lemma}
\newtheorem{cor}[thm]{Corollary}
\newtheorem{prop}[thm]{Proposition}
\newtheorem{ex}[thm]{Example}

\newtheorem{rem}[thm]{Remark}

\theoremstyle{definition}
\newtheorem{dfn}[thm]{Definition}
\newtheorem{nota}[thm]{Notation}

\def\np{\star}

\title{Orthogonality for $(0,-1)$ tropical normal matrices}
\author{Bakhad Bakhadly$^{1}$, Alexander Guterman$^{1,2}$, Mar\'{i}a Jes\'us de la Puente$^{3}$}
\date{\small $^{1}$Department of Mathematics and Mechanics,
	Lomonosov Moscow State University, Moscow, 119991, Russia \\$^{2}$  Moscow Institute of Physics and Technology, Dolgoprudny, 141701, Russia \\ $^{3}$ Departamento de \'{A}lgebra, Geometr\'{\i}a y Topolog\'{\i}a,
	Facultad de Matem\'{a}ticas,
	Plaza de Ciencias, 3,
	Universidad Complutense UCM,
	28040--Madrid, Spain}

\begin{document}

\maketitle

\begin{abstract}
We study  pairs of mutually orthogonal  normal matrices with   respect to  tropical multiplication.  Minimal orthogonal pairs
are characterized.  The diameter and girth of three graphs  arising from  the  orthogonality equivalence relation
are computed.

{\bf Keywords}: semirings, normal matrices,
orthogonality relation,  graphs, tropical algebra
\end{abstract}

\section{ Introduction}
By tropical linear algebra we mean linear algebra done with the tropical operations $a\oplus b:= \max\{a,b\}$ and $a\odot b:= a+b$.
The operations $\oplus$, $\odot$
are called \emph{tropical sum},  \emph{tropical multiplication}, respectively. These tropical operations extend,  in a natural way,
to  matrices of any order.

We work over $R=\{0,-1\}$, where we define   $(-1)\odot a= -1=a\odot (-1)$ and $0\odot a=0+a=a=a+0=a\odot 0$, for  $a\in R$  (so, zero is the neutral element with respect to $\odot$).
In particular, $(-1)\odot(-1)=(-1)+(-1)=-1$.
For tropical addition, the neutral element is  $-1$ and no opposite elements exist. To compensate this lack, tropical addition
is  idempotent:  we have  $a\oplus a=a$, for  $a\in R$. Further, we have an  order relation $-1 < 0$ compatible with the operations.
Summing up,  $(R, \oplus, \odot)$  is an \emph{ordered semiring, additively idempotent.} An additively idempotent semiring is called a dioid in \cite{Gondran}.

Note that in the definition of semiring some authors impose the condition that the neutral elements for addition and multiplication
be mutually different, but we do not. Some other authors impose that the neutral element  for addition $e$ is absorbing,
i.e., multiplication by this element is trivial $ea=e=ae, \forall a$, but we do not. Why?
Citing Pouly, \emph{ordered idempotent semirings are essentially different from fields} and
\emph{this is one reason why mathematicians are interested in semirings}; see \cite{Pouly}.  Another reason is that we want  to produce
new semirings from a given semiring, such as  the semiring of square matrices and the semiring of polynomials over the initial semiring.

We refer to  $(R, \oplus, \odot)$ as the  {\em tropical semiring} or  {\em max--plus semiring}; see \cite{Akian_HB} for a summary on max--plus properties.
The so called \emph{normal matrices}, i.e.,  matrices $[a_{ij}]$
satisfying $a_{ij}\le0$ and $a_{ii}=0$ over the tropical semiring $R$ are the protagonists of this paper.
For any $n\in \N$, the set of such square matrices over $R$ is denoted by $\norma n$ and $(M_n^N,\oplus,\odot)$ happens to be a semiring.
There exist two distinguished matrices: the all zero matrix $Z_n$ and the identity matrix $I_n=(b_{ij})$, with $b_{ii}=0$, $b_{ij}=-1$, if
$i\neq j$.

Assuming $n\ge 2$, the bizarre property of $M_n^N$ is that the same element, $I_n$, is neutral for both tropical operations
$\oplus$ and $\odot$.
Here $Z_n$ is not the neutral element for tropical addition, but it keeps the  absorbing property $AZ_n=Z_n=Z_nA$, for all $A\in\norma n$.

Every normal matrix $A$ satisfies the inequalities $I_n\le A\le Z_n$ trivially, whence
 $Z_n$ is the top element and  $I_n$ is the bottom element  in $\norma n$. Further, normal matrices satisfy
\begin{equation}
A\oplus I_n=A=I_n\oplus A, \quad AI_n=A=I_nA.
\end{equation}

Orthogonality is a fundamental notion in mathematics. The purpose of this paper is to investigate pairs of
\emph{mutually orthogonal  tropical matrices}, i.e.,
to find necessary and sufficient conditions for  square  matrices  $A,B$ to satisfy
\begin{equation}\label{eqn:ortho_matrices}
A B=Z_n=B  A,
\end{equation}
where  $Z_n$ is the  matrix, whose all elements are 0.\footnote{In the case addressed  in this paper, of normal matrices,
since the neutral element for addition and the absorbing property are not  attributes of a single element,
 to call orthogonality
to the relation $AB=Z_n=BA$ is slightly questionable (but not atrocious). Another possible definition is: given matrices
$A,B\in M_n$, say that they are \emph{mutually perpendicular}
if, for every row $r$ in $A$  and every column $c$ in $B$, the maximum $\max_{i\in[n]} r_i+c_i$ is attained at least twice, and
symmetrically, for every row $r$ in $B$  and every column $c$ in $A$, the maximum $\max_{i\in[n]} r_i+c_i$ is attained at least twice.
We do not explore this definition in the present paper.} If $A^2=Z_n$, then we say that $A$ is \emph{self--orthogonal}. To
simplify, we write $AB$  for $A\odot B$,  since there is no non--tropical multiplication of matrices in this paper.

\bigskip
Mutual orthogonality of a pair $A,B$ arises   in classical
algebra (e.g.  idempotents and  projectors, where the equivalence $A=A^2 \Leftrightarrow AB=0=BA$ holds, with $B=1-A$), functional analysis (e.g., families of orthogonal polynomials and orthogonal functions) and signal theory. In neighboring disciplines, such as  statistics, economics, computer science and
physics,  orthogonal states are considered.

Combinatorial matrix theory is the  investigation of matrices using combinatorial tools (see \cite{Brualdi_Ryser,Brualdi,Haemers,Hall_Li}).
Binary relations on associative rings and semirings, and, in particular, on the  algebra  of matrices can be understood with the help of graph theory.
Indeed,  one studies the so--called {\it relation graph}, whose vertices are matrices in some set,  and  edges
show corresponding elements
under this relation. Commuting graphs and zero--divisor graphs are examples of relation graphs; they have become  classical concepts
in algebra and combinatorics.

Orthogonality appears
in  combinatorial matrix theory  and graph theory;
see~\cite{a31, Ovchin, Semrl} and references therein.
In the paper~\cite{BGM} the notion of the graph generated by the mutual orthogonality relation for
elements of an associative ring was introduced.
In the paper \cite{Linde_Puente},  the structure of the
centralizer  $\{B:AB=BA\}$ (with tropical multiplication) of a  given normal matrix $A$ was studied.
In fact, mutual orthogonality of a pair $A,B$ is a very special case of commutativity $AB=BA$. Observe that different properties of commutativity relation of matrices over semirings were intensively studied, see \cite{BeaS03} and references therein.

Semirings are widely used in discrete event systems, dynamic programming and linguistics \cite{Dolan,Golan,Pouly}. Recent attention has been paid to the   semiring  of normal matrices over $\R\cup\{-\infty\}$  with tropical  operations, in~\cite{Yu_Zhao_Zeng}.

\medskip

This paper introduces mutual  orthogonality in the semiring of square normal matrices over $(R,\oplus,\odot)$.\footnote{
We might have decided to work over the ordered idempotent semiring
of extended non--positive real numbers  $\R_{\le0}\cup\{-\infty\}$. Our choice of simpler semiring $\{0,-1\}$ is due to the fact
that we only mind whether elements vanish or not.} Our goal is
to find necessary
and sufficient conditions on $(A,B)$ for
orthogonality.
The main results are gathered in sections \ref{sec:minimal} and \ref{sec:graphs}: these are Theorems~\ref{thm:Theta} and \ref{thm:self_orthogonal_minimal}
and Corollaries~\ref{cor:suff_conditions} and ~\ref{cor:Theta} concerning minimality, as well as
Propositions~\ref{prop:girth} to \ref{prop:diam_wnl}, concerning graphs.
We depart from an easy--to--check sufficient condition for mutual orthogonality, namely, the existence of $p,q\in[n]$,
such that the $p$--th row  and the
$q$--th column of $A$ and the $p$--th column and the $q$--th row of $B$ are zero (Lemma \ref{lem:starting_point}).
Then Theorem~\ref{thm:Theta}  characterizes minimal pairs $(A,B)$ as the members of the set $\MM_{km}$ (Notation \ref{nota:MM_km}), for some $k,m\in[n]$
with $k\neq m$.
Corollary \ref{cor:Theta}  shows that the minimal number of off--diagonal zeros in  mutually
orthogonal pairs $(A,B)$ is $\Theta_n=4n-6$, for different matrices $A, B$ of  size  $n\ge 2$, $n \neq 4$.
The key concepts are  the \emph{indicator matrix $C$ of a pair $(A,B)$} as well as  three kinds of off--diagonal zeros in $C$: \emph{propagation, cost} and \emph{gift zeros}, introduced
in Definitions \ref{dfn:indicator} and \ref{dfn:gift}. It is quite obvious that zeros propagate from $A$ and $B$ to
the products $AB$ and
$BA$ and thus, to the indicator matrix $C$. However, other zeros (called cost zeros and gift zeros) pop up in $C$. It happens
that  carefully placed
zeros in $A$ and $B$ produce gift zeros in $C$,  while not so carefully placed
zeros  in $A$ and $B$ produce cost zeros in $C$. Further,  gift zeros in $C$ are the ones to maximize in number, for minimality. This is
the pivotal idea in the paper. In the special case $A=B$ when no gift zeros exist, minimality is attained by
maximizing the number of cost zeros.
In section \ref{sec:graphs}, we study a natural graph, denoted by $\ORTHO$, arising from the orthogonality relation between
normal matrices, as well as  two subgraphs. In Propositions \ref{prop:girth} to \ref{prop:diam_wnl} we find their diameters and girths.

\bigskip
 The paper is organized as follows. In Section \ref{sec:normal}, normal matrices are defined. In Section \ref{sec:properties}, general properties of mutually orthogonal tropical normal matrix
pairs are collected. In Section  \ref{sec:minimal} we compute $\Theta_n$, the minimal number of off--diagonal zeros  in  mutually
orthogonal pairs, as well as $\Theta_n^\Delta$, the the minimal number of off--diagonal zeros  in self--orthogonal matrices. Section \ref{sec:bordering} is
devoted to the construction of orthogonal pairs of big size from smaller ones, by means of   bordered matrices.
In section \ref{sec:graphs} we compute the girth and the diameter of  three graphs
related to orthogonal  pairs.   One  graph, denoted $\ORTHO$, studies the relation $AB=Z=BA$. Another graph, denoted $\VNL$,
studies  the $(p,q)$--sufficient condition stated in the paragraph above.  The third graph, denoted $\WNL$,
studies  three other sufficient conditions for orthogonality. Altogether, these are the four sufficient conditions found in Corollary \ref{cor:suff_conditions}.

\section{Normal matrices}\label{sec:normal}
Normal matrices (and a slightly weaker notion called \emph{definite matrices}) over different sets of numbers have been studied for
more than fifty years,
under different
names, beginning with Yoeli in \cite{Yoeli}.
The notion appears in connection with tropical algebra and  geometry \cite{Butkovic_M,Butkovic_Libro,Butkovic_Sch_Serg,Litvinov_ed_2,Puente_kleene,Sergeev,Sergeev_Sch_But,Yu_Zhao_Zeng}.
In  computer science they have been called DBM ({\it difference
bound matrices}). Introduced by Bellman in the 50's, DBM matrices  are widely used in software modeling \cite{Bellman,Dima,Katoen,Mine}.

Over the semiring $(R,\oplus,\odot)$ with $R=\R_{\le0}\cup\{-\infty\}$, normal matrices
have a  direct geometric interpretation  in terms of
complexes of alcoved  convex sets in $\R^n$, see \cite{Puente_kleene}.  Then mutual orthogonality
reflects how two such complexes annihilate each other. However, to explain this is beyond the scope of this paper.

\begin{dfn}[Normal, strictly normal and abnormal] \label{dfn:normal}
A square matrix $A=[a_{ij}]$
is  \emph{normal} if $a_{ij}\le0$ and all its diagonal entries equal~0. The set  of  order $n$ normal matrices is denoted by
$\norma n$. A normal matrix $A=[a_{ij}]$ is \emph{strictly normal} if  $a_{ij}<0$ for all
$i\neq j$. The set  of  order $n$ strictly normal matrices is denoted by $M_n^{SN}$. A matrix is \emph{abnormal} if it is not normal.
\end{dfn}

Clearly, $M_n^{N}$ and $M_n^{SN}$ are closed under $\oplus$ and $\odot$, and $I_n$ is the identity element for both
operations.\footnote{$(M_n^N,\oplus)$ is a semilattice (with associative, commutative and idempotent properties).} We use classical addition and substraction  of matrices, occasionally.

\begin{nota}[Elementary matrices]\label{nota:elementary} In the set $M_n$,
\begin{enumerate}
\item let $E_{ij}$ denote the matrix with the element $-1 $ in the $(i,j)$ position, and $0$ elsewhere,      \label{mat:Zij}
\item let $U_{ij}$ denote  the  matrix with  0 in the $(i,j)$ position, all diagonal entries equal to 0,  and $-1 $ elsewhere.\label{mat:Uij}
 \item  let    $U_n$ denote  the  matrix where every entry is equal to $-1 $. We write $U$ if $n$ is understood.
\item  let $Z_n$ be the all zero matrix, and $I_n$ be the identity matrix, with zeros on the diagonal, and $-1 $ elsewhere. We write $Z$ and $I$ if $n$ is understood.
\end{enumerate}
\end{nota}

\begin{rem}\label{rem:absorbing}
 Although  $Z_n$ is not neutral for $\oplus$, the zero matrix $Z_n$ is   an     absorbing element   in $\norma n$, i.e.,
\begin{equation}\label{eqn:absorbing}
AZ_n=Z_n=Z_nA, \quad A\in M_n^N.
\end{equation}
\end{rem}
\begin{rem}
 The equality \eqref{eqn:absorbing} does not hold without normality. For example, if   $A=-E_{12}$, then $AZ_n=-(E_{11}+\cdots+E_{1n})$ and $Z_nA=-(E_{12}+\cdots+E_{n2})$  (classical addition and substraction here).
\end{rem}

\section{Pairs of mutually orthogonal matrices} \label{sec:properties}
In this section, we assume $A,B\in \norma n$.  
 Recall that the definition of mutual orthogonality is
given by the expression  (\ref{eqn:ortho_matrices}).
Our goal is
to find necessary
and sufficient conditions on $(A,B)$ for  
orthogonality.

\medskip

\begin{lem} [Orthogonality for $n=1,\: 2$]\label{lem:n_2} Let $A, B\in \norma n$.
Then\footnote{Compare tropical vs. classical linear algebra: in the classical setting, it is easy to prove that for all  $n\in\N$ and $A,B\in M_n$, the equality $AB=A+B$ implies $AB=BA$.}
\begin{enumerate}
	\item If $n=1$ then $A$ and $B$ are orthogonal if and only if $A=B=0$.\label{item:n=1}
\item If $n=2$ then
\begin{enumerate}
\item $AB=A\oplus B=BA$, \label{item:one}
\item     $A=A^2$,  i.e.,   every   normal matrix of size 2 is multiplicatively idempotent,

\item    $A$ and $B$ are orthogonal if and only if $A\oplus B=Z_2$.\label{item:three}
\label{item:two}
\end{enumerate}
\end{enumerate}
\end{lem}

\begin{proof}
For Item \ref{item:n=1} we have  $Z_1=0\in R$. By the  rule of multiplication,  orthogonality   $AB=0=BA$ is equivalent to $ A+B=0=B+A$, which holds if and only if $A=0=B$, for $A,B\in R$.

Item \ref{item:one}  is a  straightforward computation.
The remaining two items follow directly from Item \ref{item:one}.

\end{proof}

Neither of the former statements holds true for abnormal matrices.
\begin{ex}
\begin{enumerate} \item $1\odot (-1) = 0 =(-1) \odot 1$, i.e., $1$ and $-1  $ are orthogonal.
\item \begin{enumerate} \item \label{ex3.2a} Let $A=\begin{bmatrix}
	0 & 1\\
	0 & 0
	\end{bmatrix}$. Then $\begin{bmatrix}
0 & 1\\
0 & 0
\end{bmatrix} \begin{bmatrix}
0 & 1\\
0 & 0
\end{bmatrix} = \begin{bmatrix}
1 & 1\\
0 & 1
\end{bmatrix} \neq \begin{bmatrix}
0 & 1\\
0 & 0
\end{bmatrix} = \begin{bmatrix}
0 & 1\\
0 & 0
\end{bmatrix} \oplus \begin{bmatrix}
0 & 1\\
0 & 0
\end{bmatrix}$   and  $A^2\ne A$ in this case.
\item  $\begin{bmatrix}
-1 & 0\\
0 & 0
\end{bmatrix} \oplus \begin{bmatrix}
0 & 0\\
0 & -1
\end{bmatrix} = \begin{bmatrix}
0 & 0\\
0 & 0
\end{bmatrix} \neq \begin{bmatrix}
0 & -1\\
0 & 0
\end{bmatrix} = \begin{bmatrix}
-1 & 0\\
0 & 0
\end{bmatrix} \begin{bmatrix}
0 & 0\\
0 & -1
\end{bmatrix}.$
\end{enumerate}
\end{enumerate}
\end{ex}

\emph{In the rest of the paper, all matrices are assumed to be normal.}

\begin{dfn}[Indicator matrix]\label{dfn:indicator}
For matrices $A,B\in \norma n$, define product matrices  $L := AB=[l_{ij}]$ and $R := BA=[r_{ij}] \in M_{n}^N$.
The matrix $C:=[c_{ij}]\in \norma n$ given by  $c_{ij}=\begin{cases}
0, \text{\ if\ } l_{ij}=r_{ij}=0,\\
-1, \text{\   otherwise}\\
\end{cases}$  is called   the
{\em indicator matrix\/} of the pair $(A, B)$.
\end{dfn}
Obviously, the matrices $A,B$ are mutually orthogonal if and only if the indicator matrix $C$ is zero.

\medskip
The next Lemma shows how easily zeros are propagated from $A$ or $B$ to $C$, by tropical multiplication.
\begin{lem}[Propagation of zeros]\label{lem:propagation}
Let   $A, B\in \norma n$ and $p,  q\in [n]$. If $a_{pq}=0$, then
$l_{pq}=r_{pq}=c_{pq}=0$.
\end{lem}
\begin{proof}
 If $p=q$, the statement is true, by normality.
Assume now that $p\neq q$. Then
$l_{pq}=\max_{s\in [n]} ( a_{ps}+b_{sq})\le0$, and this maximum is attained at $s=q$, giving $l_{pq}=0+0=0$.
Similarly true for $r_{pq}$ and, as a  consequence, true for $c_{pq}$.
\end{proof}

\begin{cor}[Orthogonality by propagation]\label{cor:propagation}
Let  $A, B\in \norma n$. If $A\oplus B=Z_n$, then $C=Z_n$. \qed
\end{cor}

\begin{ex} [Easy orthogonality 1]  \label{ex:easy_1}
Let $A,B\in M_n^N$ be  matrices such that the number of zeros in every row and every column of each matrix is strictly greater than $n/2$. Then $A$ and $B$ are mutually orthogonal.
\emph{Indeed,  by symmetry, it is enough to prove $L=AB=Z$.
Let $r=[r_i]$  (resp. $l=[l_j]$) be an arbitrary row  of $A$ (resp. an arbitrary column  of $B$). Since multiplication in the tropical semiring is the sum, addition is the maximum, and all $r_i,l_j$ are non-positive, one obtains that  $rl=0$ if and only if there exists  $i\in [n]$ with $r_i=l_i=0$. Since more than the half of the entries of $r$ and $l$ are zero, such $i$ always does exist.}
\end{ex}

Below we show that the hypotheses of  Example \ref{ex:easy_1} are indispensable.

\begin{ex}
\begin{enumerate}
\item  \emph{The condition that the number of zeros in any row of $A$ is strictly greater than $n/2$ (only for the rows) is not sufficient for the orthogonality. The same holds only for the columns.}  Indeed, with Notations \ref{nota:elementary}, for  the matrix
  $A:=E_{1n}+\cdots + E_{n-1,n}\in\norma n$,
the entry $(1,n)$ of $A^2$ is equal to $-1$, so $A$ is not self--orthogonal.
\item \emph{The condition that exactly the half of the entries of all rows and columns of $A$ are zero is not sufficient for the orthogonality.} Indeed,  for  $n=2k$ consider the matrix    $A:=\left[ \begin{array}{cc}
Z&U\\
U&Z
\end{array}\right]\in \norma n$ .   Then the entry $(1,n)$ of $A^2$ is $-1$, so $A$ is not self--orthogonal.
\end{enumerate}
\end{ex}

\begin{ex}[Easy orthogonality 2]\label{ex:easy_2}
 Let $n\ge4$, $A=[a_{ij}], B=[b_{ij}]\in M_n^N$ and $$a_{ij}=\begin{cases}
 0,\mbox{\ if\ \ } i=j \mbox{\ \ or\ \ } i+j\equiv2 \mod 3,\\
 -1,\text{\ otherwise,}
 \end{cases}
 b_{ij}=\begin{cases}
 0,\mbox{\ if\ \ } i+j\equiv  0 \mod 2,\\
 -1,\mbox{\ otherwise.}
 \end{cases}$$
 Then $A$ and $B$ are orthogonal. Notice that $AB=Z$ follows from the facts:
 \begin{enumerate}
 \item in each row of $A$ there is zero in an odd position and there is zero in an even position,
 \item if $j$ is even (resp. odd), then $b_{kj}=0$ for all even (resp. odd) $k$.
 \end{enumerate}
 Similarly $BA=Z$.
 \end{ex}

\begin{nota}[$\nu(A)$]\label{nota:nu}
For $A\in \norma n$, let $\nu(A)$ denote the number of zero entries in $A$. Since the diagonal of $A$ vanishes, we have $\nu(A)\ge n$.
\end{nota}

Here we use the big $O$ notation. In Corollary \ref{cor:propagation} and Examples \ref{ex:easy_1} and \ref{ex:easy_2},  the number of zeros in $A$ and $B$ is $O(n^2)$.
Indeed,  $\nu(A)\ge \frac{n^2}{2}$ and $\nu(B)\ge\frac{n^2}{2}$
in Example \ref{ex:easy_1}, and  $\nu(A)$
is $O(n^2)$  and $\nu(B)=\frac{n^2}{2}$ in Example \ref{ex:easy_2}.
We want to achieve orthogonality with fewer zeros, only $O(n)$.
Our  starting point  is the  following Remark (see
Lemma \ref{lem:starting_point} for a proof).

\begin{rem} [A sufficient condition for orthogonality] \label{rem:starting_point}
Let there exist $p,q\in[n]$,
such that the $p$--th row and the $q$--th column of $A$ and the $p$--th column and the
$q$--th row of $B$  are zero. Then $AB=Z_n=BA$.
\end{rem}

Below we introduce
the convenient notation $V(p;q)$ to express Remark \ref{rem:starting_point} in short.

\begin{nota}[$V(p;q)$, $W(p;q)$ and $Z(p;q)$]\label{nota:Vs}
 For $p,q\in[n]$,  we consider  some subsets of $M_n^N$
 \begin{enumerate}
\item $V(p;q):=\{A \in   \norma n: a_{pi}=0=a_{iq},\  i\in[n]\}$,
\item $W(p;q):=\{A\in \norma n: a_{pi}=0,\  i\in[n] \setminus \{q\} \}\cap \{A\in \norma n: a_{iq}=0, \  i\in[n]\setminus \{p\} \}$,
\item $Z(p;q):=\{A\in \norma n: a_{pq}=0\}$,
\item Let $V(\sucen ps;\sucen qt)$  be the intersection of all $V(p;q)$, with  $p\in \{\sucen ps\}\subseteq [n]$ and $q\in \{\sucen qt\}\subseteq [n]$.\label{item:intersection}
 \end{enumerate}
\end{nota}

\begin{rem}\label{rem:Vs}
It is straightforward from the definitions that
\begin{equation}\label{eqn:cases_of_equ}
V(p;q)=W(p;q)\cap Z(p;q),\qquad Z(p;p)=M_n^N \qquad \text{ and }\qquad  V(p;p)=W(p;p).
\end{equation}
\end{rem}
Next we restate and prove Remark \ref{rem:starting_point}.
\begin{lem}[A sufficient condition for orthogonality]\label{lem:starting_point}
Let
$A, B\in \norma n$ and $p,q\in[n]$.
If $A\in V(p;q)$ and $B\in V(q;p)$, then  $AB=Z_n=BA$.
\end{lem}
\begin{proof}
We will prove that the indicator matrix $C$ of the pair $(A,B)$ is equal to zero.
By Lemma \ref{lem:propagation}, we have $C\in V(p;q)\cap V(q;p)$.

Consider $i,j\in [n]\setminus\{p,q\}$ with $i\neq j$. Using $A\in V(p;q)$, we get
\begin{equation}\label{eqn:AB}
0\ge l_{ij}=\max_{t\in[n]} (a_{it}+b_{tj})\ge \max_{t=q, t=i} (a_{it}+b_{tj})=\max \{b_{ij},b_{qj}\}=b_{ij}\oplus b_{qj},
\end{equation}
\begin{equation}\label{eqn:BA}
0\ge r_{ij}=\max_{t\in[n]} (b_{it}+a_{tj})\ge \max_{t=j, t=p} (b_{it}+a_{tj})=\max \{b_{ij},b_{ip}\}=b_{ij}\oplus b_{ip}.
\end{equation}
Using $B\in V(q;p)$ we get $b_{qj}=b_{ip}=0$, whence the right hand sides of (\ref{eqn:AB}) and (\ref{eqn:BA}) are zero.
Thus $l_{ij}= r_{ij}=0$ whence $c_{ij}=0$.
\end{proof}

\begin{cor}\label{cor:self_orth}
If $A\in V(p;p)$ for some $p\in [n]$, then $A^2=Z_n$. \qed
\end{cor}

\begin{dfn}[Genericity]\label{dfn:generic}
Let $S\subseteq \norma n$ be a subset. The matrix $A\in S$
is  called \emph{$S$--generic}  if   no entry of $A$ is zero, unless it is required by the structure of $S$.
\end{dfn}
 Recall that $R=\{0,-1\}$.
\begin{dfn}
For $p,q\in[n]$, the \emph{indicator function of the pair $(p,q)$ of indices} is $c(p,q)=\begin{cases}
0,& \text{\ if\  } p\neq q\\
-1,&\text{\ otherwise.}
\end{cases}$
\end{dfn}

\begin{lem}\label{lem:nu} With $p,q\in[n]$ and Notation~\ref{nota:nu},
\begin{enumerate}
\item  there exists a unique $V(p;q)$--generic  matrix  $A\in M_n^N$ and
$\nu(A)-n=2n-3-c(p,q)$,

\item there exists a unique $W(p;q)$--generic  matrix  $A\in M_n^N$ and
$\nu(A)-n=2n-4-2c(p,q),$

\item  there exists a unique $W(p;q)\cap Z(p;q)\cap Z(q;p)$--generic  matrix  $A\in M_n^N$ and $\nu(A)-n=
 2n-2.$
\end{enumerate}
\end{lem}
\begin{proof}
1. We have a zero row, a zero column and the zero main  diagonal in $A$, and the remaining entries are $-1$.	

Items 2 and 3 are similar.
\end{proof}

\begin{lem}\label{lem:kind_of_converse}
Let 
$A, B\in \norma n$ and  $p,q\in[n]$ with $p\neq q$. If $A$  is $V(p;q)$--generic and  $AB=Z_n=BA$,
then  (\ref{eqn:AB}) and (\ref{eqn:BA}) are equalities.  In particular, $b_{ij}=0$ or $b_{qj}=0=b_{ip}$, all $i,j\in[n]$.
\end{lem}

\begin{proof}
By hypothesis,  the indicator matrix   $C$ of the pair $(A,B)$ is zero, which implies $L=AB=Z=BA=R$.
The thesis is trivial, if $i,j,p,q$ are not pairwise different, by normality.

Suppose now that $i,j,p,q\in[n]$ are pairwise different.
The equality $L=Z$ implies $0=l_{ij}$.
By genericity, we know that $A\in V(p;q)$ and $a_{ij}<0$,  whenever $i\neq j$ and $i\neq p$ and $j\neq q$.  Then, in
$0=l_{ij}=\max_{t\in[n]} (a_{it}+b_{tj})$, every term on the right hand side is strictly negative except, perhaps, for $t=i$ or $t=q$.
It follows that the maximum, which is zero,  is attained at $t=i$ or $t=q$ and, furthermore,  $b_{ij}=0$ or $b_{qj}=0$. We conclude that $b_{ij}\oplus b_{qj}=0$ and (\ref{eqn:AB}) is a chain  of equalities. Similarly, we prove $b_{ij}\oplus b_{ip}=0$.
\end{proof}
In Lemma \ref{lem:kind_of_converse} we have proved that mutual orthogonality of $A$ and $B$ together with $O(n)$ \emph{aligned zeros} in
$A$ force some entries in $B$ to vanish.
\label{fact:key_for_gift} This  key observation leads us to  the notions  of cost and gift zeros given below.

\section{Minimal number of zeros in pairs}
\label{sec:minimal}
 In this section we assume  $A,B\in \norma n$.
Our goal is
to find necessary
and sufficient conditions on the pair $(A,B)$ for minimal
orthogonality, i.e.,  the orthogonality with matrices $A$  and $B$ having minimal number of zeros.

\begin{dfn}\label{dfn:minimal_pair}
Let $\Theta_n$ be the minimal number of off--diagonal zero entries among all pairs of  mutually orthogonal matrices in $\norma n$.  With Notation \ref{nota:nu},
$$\Theta_n=\min\limits_{A, B\in \norma n}\{\nu(A)+\nu(B)-2n \ :\  AB=Z_n=BA\}.$$
Let $A, B \in \norma n$.
The pair $(A, B)$ is called  {\em minimal\/} if it
realizes the value of $\Theta_n$, i.e., if
$\Theta_n=\nu(A)+\nu(B)-2n$   and $AB=Z_n=BA$.
\end{dfn}

\begin{rem}\label{rem:upper_bound}
We have $\Theta_n\le n^2 - n$,  since the pair $(Z_n,B)$, with $B$ strictly normal, satisfies
$\nu(Z_n)=n^2$ and $\nu(B)=n$.
\end{rem}

\begin{lem}\label{lem_dfn:gift}
Let $A=[a_{ij}], B=[b_{ij}]\in \norma n$
and let $C=[c_{ij}]$ be the indicator matrix of $(A,B)$.   If  $s,t,k,m\in~[n]$ are pairwise different integers such that  $a_{sk}=b_{kt}=b_{sm}=a_{mt}=0$,
then  $c_{sk}=c_{kt}=c_{sm}=c_{mt}=c_{st}=0$.
\end{lem}
\begin{proof}
We get  $c_{sk}=c_{kt}=c_{sm}=c_{mt}=0$,  by Lemma~\ref{lem:propagation}, and get   $c_{st}=0$,
by the definition of tropical multiplication.
\end{proof}

The next definition classifies the entries of indicator matrices.
\begin{dfn}[Propagation,  cost and gift zeros]\label{dfn:gift}
Let $A=[a_{ij}], B=[b_{ij}]\in \norma n$ and  $C=[c_{ij}]$ be the indicator matrix of $(A,B)$.
Let $s,t\in[n]$ with $s\neq t$.
\begin{enumerate}
\item If $a_{st}=0$ or $b_{st}=0$, then $c_{st}=0$ is called a {\em propagation zero\/}. Let  $\propa(C)$ denote the number of propagation zeros in   $C$. \label{item:propa}
\item \label{item:cost_def}If there exists $k\in [n]$ such that $s,t,k$ are pairwise different integers, and $a_{st}\neq 0\neq b_{st}$, and $a_{sk}=b_{kt}=b_{sk}=a_{kt}=0$, then $c_{st}=0$ is called a {\em cost zero\/}. For each such $k$ we can use the notation
\begin{equation}\label{eqn:cost}
c_{st}:=\phi_{st}^{kk}.
\end{equation}
\item \label{item:gift_def}If $c_{st}$ is not a cost zero and there exist $k,m\in [n]$ such that $s,t,k,m$ are pairwise different integers, and $a_{st}\neq 0\neq b_{st}$, and $a_{sk}=b_{kt}=b_{sm}=a_{mt}=0$, then $c_{st}=0$ is called a {\em gift zero\/}. For each such $k, m$ we can use the notation
\begin{equation}\label{eqn:gift}
c_{st}:=\phi_{st}^{km}.
\end{equation}
Let  $\gift(C)$ denote the number of gift zeros in~$C$.
\end{enumerate}
\end{dfn}

In plain words,  assume $a_{st}\neq 0\neq b_{st}$. Then \emph{two} zero entries ($a_{sk}=a_{mt}=0$)  in   $A$, together with
\emph{two} zero entries ($b_{kt}=b_{sm}=0$)  in   $B$ provide \emph{five} zero entries in   $C$:  a \emph{gift zero}  in the
$(s,t)$ position  and four \emph{propagation zeros}    in the $(s,k)$, $(k,t)$, $(s,m)$ and $(m,t)$ positions.
In particular, \emph{carefully placed propagation zeros are attached  to gift zeros, and conversely.}

\begin{rem}
The only difference between Items~\ref{item:cost_def} and \ref{item:gift_def} in Definition \ref{dfn:gift} is whether $k=m$ or not. Moreover, the notation $c_{st}=\phi_{st}^{km}$ is general and, in the case $k = m$, it means a cost zero.
\end{rem}

\begin{dfn}[Duplicates]\label{dfn:duplicates}
If $A\neq B$, a \emph{duplicate} in the pair $(A,B)$ is a position $(s,t)$ with $a_{st}=0=b_{st}$ and $s\neq t$.
\end{dfn}

\begin{nota}[$\Sigma(A,B,i)$ and $\Sigma(A,B)$]
Denote by $\ro A i$ the $i$--th row of the matrix $A$, and by $\nur A i$ the number of zeros in the $i$--th row of $A$.
For a given pair $(A,B)$ with $A\neq B$, denote by $\Sigma(A,B,i)$ the sum $\nur A i+\nur B i-2$ and by $\Sigma(A,B)$ the sum $\nu(A)+\nu(B)-2n$.
We simply write $\Sigma(i)$ or $\Sigma$, whenever the pair $(A,B)$ is understood from the context.
Note that $\Sigma(A,B,i)$ stands for the number of off--diagonal zeros in the $i$--th row of $A$ and $B$. With these notations $\Theta_n$ from Definition~\ref{dfn:minimal_pair} transforms into $$\Theta_n=\min\limits_{A, B\in \norma n}\{\Sigma(A,B) \ :\  AB=Z_n=BA\}.$$
\end{nota}

\begin{rem}\label{rem:exclusive}
	\begin{enumerate}
		\item Propagation, cost  and gift zeros are zeros in the indicator matrix $C$ of a pair $(A,B)$.
		\item The indicator matrix $C$ of a pair of normal matrices is normal and, therefore, $C$ has zero diagonal.
		So zeros in $C$ can
		be either diagonal, propagation, cost  or gift, and these are  mutually exclusive variants. \label{item:exclusive}
		\item  It follows from Definition \ref{dfn:gift}
that the number of cost zeros in any row of $C$ is less
than 		or equal to $n-2$, and 		
the number of propagation  zeros in  any row of $C$   is at least $1$,
 if a cost zero exists in that row. We say that $\ro C i$ is a cost row if it contains $n-2$ cost zeros and one propagation zero. Similar for columns. 		 \label{item:exclusive_cost}
		\item  It follows from Definition \ref{dfn:gift} that the number of gift zeros in any row of $C$ is less than
		or equal to $n-3$, and
		the number of propagation  zeros in  any row of $C$
 is at least $2$,
if a gift zero exists in that row. We say that $\ro C i$ is a gift row if it contains $n-3$ gift zeros and 2 propagation zeros and $\Sigma(i) = 2$. Similar for columns.
		\label{item:exclusive_gift}
        \item Gift zeros do not exist when $A=B$. Gift zeros do not exist when $n\le3$.\label{item:gift_dont_exist}
		\item A cost zero $c_{st}=\phi_{st}^{kk}$ requires $2$ duplicates $(s,k)$ and $(k,t)$ in the pair $(A, B)$.
        Note that $\Sigma(s)\ge2$ and $\Sigma(k)\ge2$.
		\label{item:2_duplicates}
		\item It is possible to have $c_{st}=\phi_{st}^{km}=\phi_{st}^{k'm'}$,
		with  $(k,m) \ne (k',m')$
		(valid for cost and gift zeros.)
		\label{item:diff_k_m}
\item If $\ro C s$ is a {cost row}, then there exists $k\in [n] \setminus \{s\}$ such that $c_{st}=\phi_{st}^{kk}$, for
all $t \in [n] \setminus \{s, k\}$. (Indeed, if for some $t_1, t_2, k_1, k_2 \in [n]$, with $k_1 \neq k_2$, we have $c_{st_1}=\phi_{st_1}^{k_1k_1}$ and $c_{st_2}=\phi_{st_2}^{k_2k_2}$, then $c_{sk_1}, c_{sk_2}$ are propagation zeros,
but $\ro C s$ contains only one propagation zero, contradiction.) We say that  the row is  {$k$--cost} to indicate
this dependence on $k$. \label{item:cost_row}
\item If  $\ro C s$ is a  {gift row},  then there exist $k,m\in [n]\setminus \{s\}$ with $k\neq m$ such
that $c_{st}=\phi_{st}^{km}$, for all $t \in [n] \setminus \{s,k,m\}$. (Similar proof to Item~\ref{item:cost_row}). We say that
the row is {$km$--gift} to indicate this dependence on $k$ and $m$. \label{item:gift_row}
\end{enumerate}
\end{rem}

\begin{nota}
Let propagation zeros be marked {\color{blue} blue} and diagonal zeros be marked {\color{red}red}, for a better visualization.
The symbol $\np$ denotes any element in $R=\{0,-1\}$.
\end{nota}

\begin{ex} \label{ex:1}
Consider any pair $(A,B)$ (not necessarily orthogonal) in $\norma 6$. Taking $(k,m)=(3,4)$ and  assuming
$a_{s3}=b_{3t}=b_{s4}=a_{4t}=0$ for different values of $s,t\in[6]$, we get different indicator matrices $C, C', C''$ of the pair $(A,B)$, where $C$ corresponds to the choice $s=1$ and $t\in \{2,5\}$, $C'$ corresponds to $s=1$ and  $t\in \{2,5,6\}$, and  $C''$ corresponds to $s\in \{1,5,6\}$ and $t=2$.
$$C=\begin{bmatrix}
{\color{red}0}& \phi_{12}^{34} & {\color{blue}0}& {\color{blue}0}& \phi_{15}^{34} &\np \\
\np  &{\color{red}0}& \np  & \np  & \np  & \np \\
\np  & {\color{blue}0}&{\color{red}0}& \np  & {\color{blue}0}& \np \\
\np  & {\color{blue}0}& \np  &{\color{red}0}& {\color{blue}0}& \np \\
\np  & \np  & \np  & \np  & {\color{red}0}& \np \\
\np  & \np  & \np  & \np  & \np  & {\color{red}0}
\end{bmatrix}, \quad  C'=\begin{bmatrix}
{\color{red}0}& \phi_{12}^{34} & {\color{blue}0}& {\color{blue}0}& \phi_{15}^{34} & \phi_{16}^{34}\\
\np  &{\color{red}0}& \np  & \np  & \np  & \np \\
\np  & {\color{blue}0}&{\color{red}0}& \np  & {\color{blue}0}& {\color{blue}0} \\
\np  & {\color{blue}0}& \np  &{\color{red}0}& {\color{blue}0}& {\color{blue}0} \\
\np  & \np  & \np  & \np  &{\color{red}0}& \np \\
\np  & \np  & \np  & \np  & \np  &{\color{red}0}
\end{bmatrix}, \quad C''=\begin{bmatrix}
{\color{red}0}& \phi_{12}^{34} & {\color{blue}0}& {\color{blue}0}& \np  & \np \\
\np  &{\color{red}0}& \np  & \np  & \np  & \np \\
\np  & {\color{blue}0}&{\color{red}0}& \np  & \np  & \np \\
\np  & {\color{blue}0}& \np  &{\color{red}0}& \np  & \np \\
\np  & \phi_{52}^{34} & {\color{blue}0}& {\color{blue}0}&{\color{red}0}& \np \\
\np  & \phi_{62}^{34} & {\color{blue}0}& {\color{blue}0}& \np  &{\color{red}0}
\end{bmatrix}.$$
\end{ex}

\begin{rem}[Bounds]\label{rem:bounds}
If $A\neq B$ then $\max\{\nu(A),\nu(B)\}-n\le \propa(C)\le \Sigma(A,B)$ holds for arbitrary pairs $(A,B)$, the second inequality being an
equality if and only if no duplicates exist.  This follows from Item \ref{item:propa} in Definition \ref{dfn:gift}. To get minimal pairs one has to avoid duplicates, as much as possible, i.e., one has to  minimize the gap between $\propa(C)$ and $\Sigma(A,B)$. If $A=B$ then $\nu(A)-n=\propa(C)$.
\end{rem}

\begin{cor}\label{cor:Theta_estimation} If  $n \ge 2$ then $\Theta_n \le 4n-6$.
\end{cor}
\begin{proof} Consider a $V(p;q)$--generic matrix $A$  and a $V(q;p)$--generic matrix $B$, with different $p,q\in[n]$, and apply
Lemmas~\ref{lem:starting_point}  and  \ref{lem:nu}.
\end{proof}

Below we introduce the set $\MM_{km}$ which we need for the later description of minimal pairs.
\begin{nota}[$\MM_{km}$]\label{nota:MM_km} For $k, m \in [n]$,  a pair $(A,B)$  belongs to $\MM_{km}$ if
\begin{enumerate} \setcounter{enumi}{-1}
\item $A$ is $V(m;k)$--generic and $B$ is $V(k;m)$--generic, or \label{item:min0}
\item $A$ is $W(m;k)\cap Z(k;m)$--generic and $B$ is $W(k;m)\cap Z(m;k)$--generic, or \label{item:min1}
\item $A$ is $W(m;k)$--generic and $B$ is $W(k;m)\cap Z(m;k)\cap Z(k;m)$--generic, or \label{item:min2}
\item $A$ is $W(m;k)\cap Z(m;k)\cap Z(k;m)$--generic and $B$ is $W(k;m)$--generic. \label{item:min3}
\end{enumerate}
\end{nota}

\begin{rem}\label{rem:MM_km}
Two special cases arise:
\begin{enumerate}
    \item If $k = m$, then, by expressions (\ref{eqn:cases_of_equ})  in Remark \ref{rem:Vs}, the four cases in \ref{nota:MM_km} reduce to one case: $A,B$ are $V(k;k)$--generic,\label{item:k=m}
    \item If $A=B$, then, by expressions (\ref{eqn:cases_of_equ}) in Remark \ref{rem:Vs} and Item~\ref{item:intersection} of Notation~\ref{nota:Vs}, the four cases in \ref{nota:MM_km} reduce to one case: $A$ is $V(k,m;k,m)$--generic.
\end{enumerate}
\end{rem}

In the following Lemma the necessity is trivial, however, sufficiency is crucial because it tells us how to recover the pair $(A,B)$ from the indicator matrix $C$.
\begin{lem} [Characterization of $\MM_{km}$]\label{lem:MM_km} Let $C=[c_{ij}]\in \norma n$ be the indicator matrix of a pair $(A,B)$ with $A\neq B$. Then $(A,B) \in \MM_{km}$ for some $k\neq m$, if and only if the following conditions hold:
\begin{enumerate}[I.]
\item for each $s,t \in [n] \setminus \{k, m\}$ with $s \neq t$, the  $(s,t)$ entry is a gift zero with $c_{st}=\phi_{st}^{km}$,  \label{item:aa}
\item $c_{km}$ and $c_{mk}$ are propagation zeros, \label{item:bb}
\item there  are no duplicates in the pair $(A,B)$ (or, equivalently, $\Sigma(A,B)= 4n-6$).\label{item:cc}
\end{enumerate}
\end{lem}
\begin{proof}
Necessity follows from Definition~\ref{dfn:gift}, Notation~\ref{nota:MM_km} and Item \ref{item:2_duplicates} in
Remark \ref{rem:exclusive}. To prove sufficiency, notice that
Items~\ref{item:bb} and \ref{item:cc} initialize four cases (both zeros come from $A$ and $a_{km}=0=a_{mk}$, or
both zeros come from $B$ and $b_{km}=0=b_{mk}$, or $a_{km}=0=b_{mk}$  or $a_{mk}=0=b_{km}$). Then, by Item \ref{item:aa},
we get $a_{sk}=b_{kt}=b_{sm}=a_{mt}=0$ for all $s,t\in[n]\setminus\{k,m\}, s\neq t$, and no further  off--diagonal zeros
appear in $A$ or $B$, by Item \ref{item:cc}. Thus,  we get cases \ref{item:min0} to \ref{item:min3} of Notation \ref{nota:MM_km}.
An illustration of how the argument works is found in Example~\ref{ex:2}.
\end{proof}
\begin{cor}\label{cor:MM_km} For each pair  $(A,B) \in \MM_{km}$ with
$A \neq B$ and $k\neq m$, the indicator matrix $C$ satisfies
\begin{enumerate}
\item $C = Z_n$,  \label{item:C_null}
\item $\propa(C)=\Sigma(A,B)=4n-6$  and $\gift(C)=(n-3)(n-2)$.\label{item:MM_4n-6}
\end{enumerate}
\end{cor}
\begin{proof}
For Item \ref{item:C_null}, we must prove orthogonality of $(A,B)$.
If the pair $(A,B)$ is in case \ref{item:min0} of Notation \ref{nota:MM_km}, then orthogonality holds true, by Lemma \ref{lem:starting_point}.
If the pair $(A,B)$ is in cases \ref{item:min1} to \ref{item:min3}, then orthogonality is checked, similarly.
Now, to prove Item \ref{item:MM_4n-6}, we count gift zeros using Item~\ref{item:aa} and then count propagation zeros using Items~\ref{item:bb} and \ref{item:cc} in Lemma~\ref{lem:MM_km}.
\end{proof}
\begin{ex}\label{ex:2} The following four items correspond to the $4$ items from Notation~\ref{nota:MM_km}
for $(k,m)= (4,3)$. Notice that the differences between the items occur only in entries $(k,m)$ and $(m,k)$. That is why
the indicator matrix $C$ is the same in all cases. Below the symbol "$-$"    denotes the element  $-1\in R$.
$$C = \begin{bmatrix}
{\color{red}0}& \phi_{12}^{43} &{\color{blue}0}& {\color{blue}0}& \phi_{15}^{43} & \phi_{16}^{43}\\
\phi_{21}^{43} &{\color{red}0}&{\color{blue}0}& {\color{blue}0}& \phi_{25}^{43} & \phi_{26}^{43}\\
{\color{blue}0}&{\color{blue}0}&{\color{red}0}&{\color{blue}0}& {\color{blue}0}&{\color{blue}0}\\
{\color{blue}0}& {\color{blue}0}&{\color{blue}0}&{\color{red}0}& {\color{blue}0}& {\color{blue}0}\\
\phi_{51}^{43} & \phi_{52}^{43} &{\color{blue}0}&{\color{blue}0}&{\color{red}0}& \phi_{56}^{43}\\
\phi_{61}^{43} & \phi_{62}^{43} &{\color{blue}0}&{\color{blue}0}& \phi_{65}^{43} & {\color{red}0}
\end{bmatrix}.$$

$$0.\ \ A_0=\begin{bmatrix}
0 & - & - & 0 & - & -\\
- & 0 & - & 0 & - & -\\
0 & 0 & 0 & 0 & 0 & 0\\
- & - & - & 0 & - & -\\
- & - & - & 0 & 0 & -\\
- & - & - & 0 & - & 0
\end{bmatrix}, \
B_0=\begin{bmatrix}
0 & - & 0 & - & - & -\\
- & 0 & 0 & - & - & \\
- & - & 0 & - & - & -\\
0 & 0 & 0 & 0 & 0 & 0\\
- & - & 0 & - & 0 & -\\
- & - & 0 & - & - & 0
\end{bmatrix}$$
$$1.\ \
A_1=\begin{bmatrix}
0 & - & - & 0 & - & -\\
- & 0 & - & 0 & - & -\\
0 & 0 & 0 & - & 0 & 0\\
- & - & 0 & 0 & - & -\\
- & - & - & 0 & 0 & -\\
- & - & - & 0 & - & 0
\end{bmatrix}, \
B_1=\begin{bmatrix}
0 & - & 0 & - & - & -\\
- & 0 & 0 & - & - & -\\
- & - & 0 & 0 & - & -\\
0 & 0 & - & 0 & 0 & 0\\
- & - & 0 & - & 0 & -\\
- & - & 0 & - & - & 0
\end{bmatrix}$$
$$2.\ \ A_2=\begin{bmatrix}
0 & - & - & 0 & - & -\\
- & 0 & - & 0 & - & -\\
0 & 0 & 0 & - & 0 & 0\\
- & - & - & 0 & - & -\\
- & - & - & 0 & 0 & -\\
- & - & - & 0 & - & 0
\end{bmatrix}, \
B_2=\begin{bmatrix}
0 & - & 0 & - & - & -\\
- & 0 & 0 & - & - & -\\
- & - & 0 & 0 & - & -\\
0 & 0 & 0 & 0 & 0 & 0\\
- & - & 0 & - & 0 & -\\
- & - & 0 & - & - & 0
\end{bmatrix}$$
$$3.\ \ A_3=\begin{bmatrix}
0 & - & - & 0 & - & -\\
- & 0 & - & 0 & - & -\\
0 & 0 & 0 & 0 & 0 & 0\\
- & - & 0 & 0 & - & -\\
- & - & - & 0 & 0 & -\\
- & - & - & 0 & - & 0
\end{bmatrix}, \
B_3=\begin{bmatrix}
0 & - & 0 & - & - & -\\
- & 0 & 0 & - & - & -\\
- & - & 0 & - & - & -\\
0 & 0 & - & 0 & 0 & 0\\
- & - & 0 & - & 0 & -\\
- & - & 0 & - & - & 0
\end{bmatrix}.$$
\end{ex}
\begin{cor}[Four sufficient conditions for the existence of orthogonal pairs] \label{cor:suff_conditions}
Let $n\ge 2$ and $A,B\in \norma n$. Then the pair $(A,B)$ is orthogonal if there exist $k,m\in[n]$  such that one of the following holds:
\begin{enumerate} \setcounter{enumi}{-1}
\item $A\in V(k;m)$ and $B\in V(m;k)$. \label{item:d}
\item $A\in W(k;m)\cap Z(k;m)\cap Z(m;k)$ and $B\in W(m;k)$. \label{item:a}
\item $A\in W(k;m)\cap Z(m;k)$ and $B\in W(m;k)\cap Z(k;m)$. \label{item:b}
\item $A\in W(k;m)$ and $B\in W(m;k)\cap Z(k;m)\cap Z(m;k)$. \label{item:c}
\end{enumerate}
\end{cor}
\begin{proof}
If $k=m$, then it follows from Item \ref{item:k=m} of Remark~\ref{rem:MM_km} and Lemma \ref{lem:starting_point}.
If $k \neq m$, then it
follows from Item~\ref{item:C_null} of Corollary~\ref{cor:MM_km}, by allowing $A$ or $B$ to have more zeros than strictly required by the structure of sets $V(k;m)$, $W(k;m)$, $Z(k;m)$. Note that Item \ref{item:d} here  is  just Lemma \ref{lem:starting_point},  the  sufficient condition we started with. 
\end{proof}

\begin{lem}\label{lem:two_zeros}
Let $n \ge 3$. If the pair $(A,B)$ is orthogonal, then $\Sigma(i) \ge 2$, for all $i \in [n]$.
\end{lem}
\begin{proof}
Suppose that $\nur A i = 1$, i.e., we have only one diagonal zero in $\ro A i$, then $\ro B i$ is zero, by definition of
tropical multiplication, hence $\Sigma(i) = n - 1 \ge 2$. Similarly if $\nur B i = 1$. If $\nur A i > 1$ and
$\nur B i > 1$, then also $\Sigma(i) \ge 2$ and the proof is complete.
\end{proof}

\begin{lem}\label{lem:3_rows}
If the pair $(A,B)$ is minimal and $n \ge 3$, then there exist at least three mutually different  indices  $i\in [n]$
such that $2 \le \Sigma(i) \le 3$.
\end{lem}
\begin{proof}
The first inequality $2 \le \Sigma(i)$ holds for all $i\in[n]$ by Lemma~\ref{lem:two_zeros}. Suppose that for $n-2$
rows $i$ we have $\Sigma(i) \ge 4$. Then $\Sigma \ge 4(n-2) + 2\cdot 2=4n-4 > 4n-6$, which contradicts
with Corollary~\ref{cor:Theta_estimation} and the proof is complete.
\end{proof}

\subsection{Arbitrary pairs}\label{subsec:different}

The aim of this subsection is to prove that $\Theta_n=4n-6$, for $n\ge2$, $n \neq 4$, as well as to \emph{construct minimal pairs}.

In the following we assume  that 
$C\in \norma n$ is the indicator matrix of the  pair $(A,B)$.

\begin{lem}\label{lem:no_row_cost} Let $n \ge 6$. If the pair  $(A,B)$ is minimal,
then no  row in $C$ is a cost row.
\end{lem}
\begin{proof}
By hypothesis, the pair $(A,B)$ is orthogonal.
Suppose there exists $s\in[n]$ such that $\ro C s$ contains $n-2$ cost zeros. Then, by Item~\ref{item:cost_row} of
Remark \ref{rem:exclusive}, there exists $k\in [n] \setminus \{s\}$ such that $c_{st}=\phi_{st}^{kk}$, for all
$t \in [n] \setminus \{s, k\}$. Then, by Item~\ref{item:2_duplicates} of Remark~\ref{rem:exclusive}, we have $n-2$ duplicates $(k,t)$, $t \in [n] \setminus \{s, k\}$, hence $\Sigma(k)\ge 2(n-2)$. Using Lemma~\ref{lem:two_zeros} we get $\Sigma \ge 2(n-1) + 2(n-2) = 4n-6$.
Hence $\Sigma=4n-6$, by Corollary~\ref{cor:Theta_estimation}. So $\Sigma(k)= 2(n-2)$ and $\Sigma(i) = 2$, for all $i \neq k$,
i.e., $(A, B)$ has the following structure based on the number of zeros: the number of off--diagonal zeros in $\ro A k$ and
$\ro B k$ is exactly $2(n-2)$ and the number of off--diagonal zeros in $\ro A i$ and $\ro B i$ is exactly $2$, for  all $i \neq k$.
We have four cases.
\begin{enumerate}
\item  (A duplicate exists.) If there exists $s'\in[n]\setminus\{s,k\}$ such that $a_{s'k'}=b_{s'k'}=0$ for some
$k'$ with $k'\neq s'$, then there is only one propagation zero
in $\ro C {s'}$.
Then by Item~\ref{item:exclusive_gift} of Remark~\ref{rem:exclusive} there is no gift zero in $\ro C {s'}$. Since $\ro C {s'}$ is
zero, then by Items~\ref{item:exclusive} and \ref{item:cost_row} of Remark~\ref{rem:exclusive},
 $\ro C {s'}$ is $k'$--cost and $c_{s't}=\phi_{s't}^{k'k'}$, for all $t \in [n]
\setminus \{s', k'\}$.
Hence $\Sigma(k') \ge 2(n-2)$. Since $2(n-2) > 2$ we get that $k' = k$. Since for $t=s$
we have $c_{s's}=\phi_{s's}^{kk}$, then, by Item~\ref{item:2_duplicates} of Remark~\ref{rem:exclusive}, $(k,s)$ is also a duplicate alongside with other $n-2$ duplicates  $(k,t)$, $t \in [n] \setminus \{s, k\}$, hence
$\Sigma(k)\ge 2(n-1) > 2(n-2)$, which is a contradiction with the structure of $(A,B)$.

\item If there exists $s'\in[n]\setminus\{s,k\}$ such that $a_{s'k'}=b_{s'm'}=0$ for some $k',m'$, where $s',k',m'$
are pairwise different, then there are only two propagation zeros
in $\ro C {s'}$. Moreover, by Item~\ref{item:2_duplicates} of Remark~\ref{rem:exclusive}, there is no cost zero in $\ro C {s'}$, since $b_{s'k'}\neq 0\neq a_{s'm'}$ (no duplicates). Since
$\ro C {s'}$ is zero, then
by Items~\ref{item:exclusive} and \ref{item:gift_row} of Remark~\ref{rem:exclusive} the $\ro C {s'}$ is $k'm'$--gift and so
$c_{s't}=\phi_{s't}^{k'm'}$, for all
$t \in [n] \setminus \{s', k', m'\}$.
Then, by
Item~\ref{item:gift_def} of Definition~\ref{dfn:gift}, $\nur A {m'} - 1 \ge n-3$ and $\nur B {k'} - 1 \ge n-3$.
Since $n-3 > 2$, we have $2$ rows with more than $2$ off--diagonal zeros, which contradicts with the structure of $(A,B)$.

\item If there exists $s'\in[n]\setminus\{s,k\}$ such that $\ro A {s'}$ has no off--diagonal zeros, then $\ro B {s'}$
is zero and $\Sigma(s') = n - 1$, which is
a contradiction with
the structure of $(A,B)$.

\item Similarly, if there exists $s'\in[n]\setminus\{s,k\}$ such that $\ro B {s'}$ has no off--diagonal zeros, we also get
a contradiction with the structure of $(A,B)$.
\end{enumerate}
Thus,  there is no row in $C$ with $n-2$ cost zeros, and the proof is complete.
\end{proof}

\begin{lem}\label{lem:2_km_zeros} Let $n \ge 4$, the pair $(A,B)$ be  orthogonal,
and $\ro C s$ be $km$--gift. Then the following hold:
\begin{enumerate} 		
\item $\Sigma(k) \ge n-2$ and $\Sigma(m) \ge n-2$.\label{item:2_km_1}
\item If there exists $s' \in [n] \setminus \{s,k\}$ such that $c_{s's}=\phi_{s's}^{km'}$ for some $m' \in [n] \setminus \{s, s'\}$,
then $\Sigma(k) \ge n-1$.\label{item:2_km_k} 		
\item If there exists   $s' \in [n] \setminus \{s,k\}$
such that $c_{s's}=\phi_{s's}^{k'm}$ for some   $k' \in [n] \setminus \{s, s'\}$, then $\Sigma(m) \ge n-1$.\label{item:2_km_m} 	
\end{enumerate}
\end{lem}

\begin{proof}
\begin{enumerate}
\item  Since $\ro C s$ is
$km$--gift, then $s,k,m$ are mutually different and, by Item \ref{item:gift_row} of Remark \ref{rem:exclusive},
it has $n-3$ gift zeros
$c_{st}=\phi_{st}^{km}$, for all $t \in[n]\setminus \{s,k,m\}$.
Thus, by Item~\ref{item:gift_def} of Definition~\ref{dfn:gift}, $a_{mt}=b_{kt}=0$, $t \in[n]\setminus \{s,k,m\}$, hence
$\nur A m - 1 \ge n-3$ and $\nur B k - 1 \ge n-3$.  Consider
$c_{km}$.  Since $(A,B)$ is an orthogonal pair, then $C = Z$, hence $c_{km} = 0$. By Item~\ref{item:exclusive}
of Remark~\ref{rem:exclusive}, we have one of the following cases: 		
\begin{enumerate}[i.] 			
\item If $c_{km}$ is a
propagation zero, then at least one of $a_{km}, b_{km}$ is zero, hence $\Sigma(k) \ge n-3+1=n-2$. 			
\item If $c_{km}$ is a
cost zero $\phi_{km}^{k'k'}$ for some $k'$ then, by Item~\ref{item:cost_def} of Definition~\ref{dfn:gift}, $a_{kk'}=0$
(a zero which has not been counted previously) and
$\Sigma(k)\ge n-3+1=n-2$. 			
\item If $c_{km}$ is a gift zero $\phi_{km}^{k'm'}$ for some $k',m'$ then, by
Item~\ref{item:gift_def} of Definition~\ref{dfn:gift}, $a_{kk'}=0$ (a zero which has not been counted previously)
and $\Sigma(k) \ge n-3+1=n-2$. 		
\end{enumerate} 	
Thus,  in each case we get $\Sigma(k) \ge n-2$. Reasoning similarly for $c_{mk} = 0$, we also get $\Sigma(m)\ge n-2$.
\item By Item~\ref{item:2_km_1}, we have $\Sigma(k) \ge n-2$. Since $c_{s's}=\phi_{s's}^{km'}$
(this can be a cost or a gift zero), then, by Items~\ref{item:cost_def}, \ref{item:gift_def} of Definition~\ref{dfn:gift}
we have $b_{ks}=0$. Note that $b_{ks}$ does not coincide with the other $n-2$ zero entries in $\ro A k$ and $\ro B k$ mentioned
above. Hence $\Sigma(k) \ge n-2 + 1 = n-1$. 		 \item By Item~\ref{item:2_km_1}, we have $\Sigma(m) \ge n-2$.
Since $c_{s's}=\phi_{s's}^{k'm}$, then by Items~\ref{item:cost_def}, \ref{item:gift_def} of Definition~\ref{dfn:gift} we have $a_{ms}=0$.
Note that $a_{ms}$ does not coincide with other $n-2$ zero entries in $\ro A m$ and $\ro B m$ mentioned above. Hence $\Sigma(m)
\ge n-2 + 1 = n-1$. 	
\end{enumerate}
\end{proof}

\begin{lem}\label{lem:1_2}
Let $n \ge 5$, $(A,B)$ be an orthogonal pair, and
$\Sigma(s) = 3$, for some $s \in[n]$. Then either $\nur A s - 1 = 1$ and $\nur B s - 1 = 2$ or
$\nur A s - 1 = 2$ and $\nur B s - 1 = 1$.
\end{lem}
\begin{proof}
Indeed, if $\nur A s - 1 = 0$, then $\ro B s$ must be zero, by
tropical multiplication, and then $\Sigma(s) = n-1$, which contradicts with $\Sigma(s) = 3$.
Similarly, if $\nur B s - 1 = 0$.
\end{proof}

\begin{lem}\label{lem:3_zeros} Let $n \ge 5$ and $(A,B)$ be an orthogonal pair. 
If $\Sigma(s) = 3$, for some $s \in[n]$, then either
\begin{enumerate}
\item $a_{sk}=b_{sl}=b_{sm}=0$ for  some $k,l,m \in[n]\setminus\{s\}$, with $l\neq m$, and for each
$t \in[n]\setminus \{s,k,l,m\}$ we have $c_{st}=\phi_{st}^{km}$ or $c_{st}=\phi_{st}^{kl}$, or \label{item:3_zeros_a}
\item $b_{sk}=a_{sl}=a_{sm}=0$ for  some $k,l,m \in[n]\setminus\{s\}$, with $l\neq m$, and
for each $t \in[n]\setminus \{s,k,l,m\}$ we have  $c_{st}=\phi_{st}^{mk}$ or $c_{st}=\phi_{st}^{lk}$. \label{item:3_zeros_b}
\end{enumerate}
In any case, there exists $k\in[n] \setminus\{s\}$ with $\Sigma(k)\ge n-3$.
\end{lem}
\begin{proof}
By Lemma \ref{lem:1_2} we have two cases:
\begin{enumerate}
    \item
    If $\nur A s - 1 = 1$ and
    $\nur B s - 1 = 2$, then $a_{sk}=b_{sl}=b_{sm}=0$, for some   $k,l,m \in[n]\setminus\{s\}$, with
    $l\neq m$. Then we have at most $3$ propagation zeros in $\ro C s$ and, by Item \ref{item:exclusive} of Remark \ref{rem:exclusive} and Items~\ref{item:cost_def},~\ref{item:gift_def} of Definition~\ref{dfn:gift}, for each  $t \in[n]\setminus \{s,k,l,m\}$ we have  $c_{st}=\phi_{st}^{km}$ or $c_{st}=\phi_{st}^{kl}$. Thus, $b_{kt} = 0$ for all $t \in[n]\setminus \{s,k,l,m\}$ and so	$\nur B k - 1 \ge n-4$. If $\nur A k - 1 = 0$,
    then $\ro B k$ must be zero and $\Sigma(k) = n-1 \ge n-3$. If $\nur A k - 1 > 0$, then also
    $\Sigma(k)\ge n-3$, and Item \ref{item:3_zeros_a} is proved.
    \item If $\nur A s - 1 = 2$ and
    $\nur B s - 1 = 1$, then $b_{sk}=a_{sl}=a_{sm}=0$, for some $k,l,m \in[n]\setminus\{s\}$, with
    $l\neq m$.  The rest of the proof is similar to the proof of the previous 
    item.
\end{enumerate}
\end{proof}

\begin{lem}\label{lem:2_rows_km}
Let $n \ge 6$. If the pair $(A,B)$ is minimal and
 there are at least two different gift rows $s,s'$ in $C$,
 then both rows are $km$--gift, for the same $k,m\in[n]\setminus\{s,s'\}$, with $k\neq m$.
\end{lem}
\begin{proof}
Using Item \ref{item:gift_row} of Remark \ref{rem:exclusive}, denote the gift zeros of $\ro C s$ by $c_{st}=\phi_{st}^{km}$, for all $t \in[n]\setminus \{s,k,m\}$, and the gift zeros of $\ro C {s'}$ by $c_{s't}=\phi_{s't}^{k'm'}$, for all $t \in[n]\setminus \{l,k',m'\}$.
We have three cases:
\begin{enumerate}
	\item Suppose that
$\{k, m\} \cap \{k', m'\}= \emptyset$. Then, by Item~\ref{item:2_km_1} of Lemma~\ref{lem:2_km_zeros} for rows $s,s'$ and by Lemma~\ref{lem:two_zeros}, we get $\Sigma \ge 4(n-2)+2(n-4)=6n-16>4n-6$, which contradicts with Corollary \ref{cor:Theta_estimation}.
	\item Suppose that
$|\{k, m\} \cap \{k', m'\}|=1$.
	Without loss of generality consider two cases: $k=k'$ and $k=m'$. If $k=k'$, then by Lemma~\ref{lem:2_km_zeros} for rows $s,s'$ and by Lemma~\ref{lem:two_zeros} we get $\Sigma \ge 2(n-2)+(n-1)+2(n-3)=5n-11>4n-6$, which contradicts with Corollary \ref{cor:Theta_estimation}.
	If $k=m'$, then, by Item~\ref{item:2_km_1} of Lemma~\ref{lem:2_km_zeros} for rows $s,s'$, by Item~\ref{item:gift_def} of Definition~\ref{dfn:gift}, and by Lemma~\ref{lem:two_zeros}, we get $\Sigma \ge 2(n-2)+2(n-3)+2(n-3)=6n-16>4n-6$, because $\nur A k - 1 \ge n-3$ and $\nur B k - 1 \ge n-3$, which contradicts with Corollary \ref{cor:Theta_estimation}.
	\item Suppose that $\{k, m\} = \{k', m'\}$. If $k = k'$ and $m = m'$ then the Lemma is proved. If $k=m'$ and $m=k'$, then by Item~\ref{item:gift_def} of Definition~\ref{dfn:gift} and by Lemma~\ref{lem:two_zeros} we get $\Sigma \ge 4(n-3)+2(n-2)=6n-16>4n-6$, because $\nur A i - 1 \ge n-3$ and $\nur B i - 1 \ge n-3$ for $i=k,m$, which contradicts with Corollary \ref{cor:Theta_estimation}.
\end{enumerate}
Thus, $k=k'$ and $m=m'$ and the proof is complete.
\end{proof}

\begin{lem}\label{lem:n-3_gift_zeros}
Let $n \ge 6$ and the pair $(A,B)$ be minimal.
If $\Sigma(i) = 2$ for some $i\in [n]$, then $\ro C i$  is gift. In particular, $A\neq B$.
\end{lem}
\begin{proof}
There are three cases:
\begin{enumerate}
 \item If  there exist
 $j,j'\in[n]\setminus\{i\}$ with $j\neq j'$ such
 that $a_{ij}=a_{ij'}=0$, then $\nu(\row(B,i))=1$, since $\Sigma(i)=2$. Then, by tropical multiplication, $\row(A,i)$ must be zero, and this contradicts $\Sigma(i)=2$.
  If there exist
 $j,j'\in[n]\setminus\{i\}$ with $j\neq j'$ such that $b_{ij}=b_{ij'}=0$, it is similar.
 \item If there exist
 $j,j'\in[n]\setminus\{i\}$ with $j\neq j'$ such that $a_{ij}=b_{ij'}=0$, then we have a pair of propagation zeros and no duplicates in the $i$--th row,  hence, by Item \ref{item:exclusive} of Remark \ref{rem:exclusive}, we have $n-3$ gift zeros in the $i$--th row of $C$, providing a gift row
 and proving the Lemma.
 \item The remaining case is $j=j'$ and $a_{ij}=b_{ij}=0$ (a duplicate), in which we have exactly one propagation zero in $\ro C i$, then, by Item~\ref{item:exclusive_gift} of Remark~\ref{rem:exclusive}, there is no gift zero in $\ro C p$. Hence, by Item \ref{item:exclusive} of Remark \ref{rem:exclusive},
 we have $n-2$ cost zeros in $\ro C i$, so that the row is cost, contradicting Lemma~\ref{lem:no_row_cost}.
\end{enumerate}
In every case we have $A$ different from $B$, due to the existence of gift zeros.
\end{proof}

The following Lemma uses Notation \ref{nota:MM_km}.

\begin{lem}\label{lem:2_gift_rows}
Let $n \ge 6$ and the pair $(A,B)$ be minimal.
If there exist at least two different $p, p'\in[n]$ with $\Sigma(p) =\Sigma(p') = 2$, then $A \neq B$ and $(A,B) \in \MM_{km}$ for some $k,m$ with
$ k \neq m$.
\end{lem}
\begin{proof}
By Lemma~\ref{lem:n-3_gift_zeros}, $A \neq B$ and rows $p$ and $p'$ of $C$ are gift rows, and
by Lemma~\ref{lem:2_rows_km},  there exist $k,m\in[n]\setminus\{p,p'\}$ with $k\neq m$ such that $c_{qt}=\phi_{qt}^{km}$, for all $t \in[n]\setminus \{q,k,m\}$  and $q = p,p'$.
Then, by Items \ref{item:2_km_k}, \ref{item:2_km_m} of Lemma~\ref{lem:2_km_zeros} for row $p$,
and, by Lemma~\ref{lem:two_zeros}, we get $\Sigma \ge 2(n-1)+2(n-2)=4n-6$. By Corollary~\ref{cor:Theta_estimation} $\Sigma = 4n-6$.
So $(A, B)$ has the following structure based on the number of zeros: $\Sigma(q) = n-1$ for $q=k,m$ and $\Sigma(q) = 2$ for
$q\neq k,m$. Then similarly we get for all $q\neq k,m$ that $c_{qt}=\phi_{qt}^{km}$, $t \in[n]\setminus \{q,k,m\}$. Show that
$c_{km}$ and $c_{mk}$ are propagation zeros. Indeed, if $c_{km} = \phi_{km}^{k'm'}$ for some $k',m' \in [n]\setminus \{k,m\}$
(not  necessarily $k'\neq m'$), then $a_{m'm}=0$, but $b_{m'm}$ is also zero because of gift zeros in $\ro C {m'}$, hence
$\Sigma(m')= 3 > 2$, which is a contradiction with the structure of $(A,B)$. Similarly for $c_{mk}$.
Hence, by Lemma~\ref{lem:MM_km}, $(A,B) \in \MM_{km}$ and the proof is complete.
\end{proof}

\begin{lem}\label{lem:1_gift_row} Let $n \ge 6$ and the pair $(A,B)$ be minimal.
 If there exists  $p\in [n]$ with  $\Sigma(p)=2$, then there exists  $p'\in [n]$ with $p\neq p'$  and  $\Sigma(p')=2$.
\end{lem}

\begin{proof}
By contradiction, assume that $\Sigma(p)=2$ holds only for $p\in[n]$. Then, by Lemma~\ref{lem:two_zeros},
$\Sigma(i)\ge 3$, for all $i\in [n]\setminus\{p\}$.
By Lemma~\ref{lem:n-3_gift_zeros},
$\ro C p$ is a gift row. Moreover, by  Item~\ref{item:2_km_1} of Lemma~\ref{lem:2_km_zeros} there are 2 rows with at least
$(n-2)$ off--diagonal zeros. Thus,  as a whole  we get $\Sigma \ge 2(n-2)+3(n-3) + 2=5n-11>4n-6$, which contradicts
with Corollary~\ref{cor:Theta_estimation}, and the proof is complete.
\end{proof}

\begin{lem}\label{lem:0_gift_rows} Let $n \ge 7$ and the pair  $(A,B)$ be minimal.
Then there exists $p\in[n]$ with $\Sigma(p) = 2$.
\end{lem}

\begin{proof}
By contradiction, assume that there is no row $p$ with $\Sigma(p) = 2$.
By Lemma~\ref{lem:3_rows}, there exists $s\in[n]$ such that
$2 \le \Sigma(s) \le 3$.
Hence, by Lemma~\ref{lem:two_zeros}, $\Sigma(i) \ge 3$ for $i\in[n]$ and $\Sigma(s) = 3$.
By Lemma~\ref{lem:3_zeros} for row $s$, we have two cases,
without loss of generality suppose that
$a_{sk}=b_{sl}=b_{sm}=0$ for  some $k,l,m \in[n]\setminus\{s\}$, with $l\neq m$, and for each
$t \in[n]\setminus \{s,k,l,m\}$ we have $c_{st}=\phi_{st}^{km}$ or $c_{st}=\phi_{st}^{kl}$, and
$\Sigma(k)\ge n-3$. Using
$\Sigma(i) \ge 3$ for $i \in[n]\setminus\{k\}$, we get $\Sigma \ge (n-3)+3(n-1)=4n-6$, hence,
by Corollary~\ref{cor:Theta_estimation}, $\Sigma =4n-6$. Then $\Sigma(k) = n-3$ and $\Sigma(i) = 3$ for
$i \in[n]\setminus\{k\}$.
Take row $s' \notin \{s,k,l,m\}$  (it exists, because $n \ge 7$). Since $s' \neq k$, then, by
Lemma~\ref{lem:3_zeros}, we have two cases:
  	\begin{enumerate}
  		\item If $a_{s'k'}=b_{s'l'}=b_{s'm'}=0$ for  some $k',l',m' \in[n]\setminus\{s'\}$, with $l'\neq m'$, and for each $t \in[n]\setminus \{s',k',l',m'\}$ we have $c_{s't}=\phi_{s't}^{k'm'}$ or $c_{s't}=\phi_{s't}^{k'l'}$, and $\Sigma(k')\ge n-3$.
  		\begin{enumerate}
  			\item If $k \neq k'$, then $\Sigma(k')\ge n-3 \ge 4$. It is a  contradiction with $\Sigma(q) = 3$ for $q\neq k$.
  			\item If $k = k'$, then $c_{s't}=\phi_{s't}^{kl'}$ or $c_{s't}=\phi_{s't}^{km'}$, for all $t \in[n]\setminus \{s',k,l',m'\}$. Since $s' \in[n]\setminus \{s,k,l,m\}$ and $c_{st}=\phi_{st}^{kl}$ or $c_{st}=\phi_{st}^{km}$, for all $t \in[n]\setminus \{s,k,l,m\}$, then $c_{ss'}=\phi_{ss'}^{kl}$ or $c_{ss'}=\phi_{ss'}^{km}$. Hence by Items~\ref{item:cost_def},~\ref{item:gift_def} of Definition~\ref{dfn:gift} $b_{kt} = 0$ for all $t \in[n]\setminus \{k,l',m'\}$, whence $\nur B k - 1 \ge n-3$. If $\nur A k - 1 = 0$, then $\ro B k$ must be zero and $\Sigma(k) \ge n-1 > n-3$. If $\nur A k - 1 > 0$, then also $\Sigma(k)> n-3$, which contradicts with $\Sigma(k) = n - 3$.
  		\end{enumerate}
  		\item If $b_{s'k'}=a_{s'l'}=a_{s'm'}=0$ for  some $k',l',m' \in[n]\setminus\{s'\}$, with $l'\neq m'$, and for each $t \in[n]\setminus \{s',k',l',m'\}$ we have $c_{s't}=\phi_{s't}^{m'k'}$ or $c_{s't}=\phi_{s't}^{l'k'}$, and $\Sigma(k')\ge n-3$.
  		\begin{enumerate}
  			\item If $k \neq k'$, then $\Sigma(k')\ge n-3 \ge 4$ contradicts with $\Sigma(q) = 3$ for $q\neq k$.
  			\item If $k = k'$, then $c_{s't}=\phi_{s't}^{l'k}$ or $c_{s't}=\phi_{s't}^{m'k}$, for all
  $t \in[n]\setminus \{s',k,l',m'\}$. Hence by Items~\ref{item:cost_def},~\ref{item:gift_def} of
  Definition~\ref{dfn:gift} $a_{kt} = 0$ for all $t \in[n]\setminus \{s',k,l',m'\}$, whence
  $\nur A k - 1 \ge n-4$. Since $c_{st}=\phi_{st}^{kl}$ or $c_{st}=\phi_{st}^{km}$, for all
  $t \in[n]\setminus \{s,k,l,m\}$, then by Items~\ref{item:cost_def},~\ref{item:gift_def} of
  Definition~\ref{dfn:gift} $b_{kt} = 0$ for all $t \in[n]\setminus \{s,k,l,m\}$, whence $\nur B k - 1 \ge n-4$.
  Hence $\Sigma(k) \ge 2(n-4) > n-3$, which contradicts with $\Sigma(k)= n - 3$.
  		\end{enumerate}
  \end{enumerate}
Thus, $(A,B)$ cannot be minimal and the proof is complete.
\end{proof}

\begin{lem}\label{lem:necessity}
Let $n=2$ or $n \ge 7$. If the pair $(A,B)$ is minimal, then $A \neq B$ and $(A,B)\in \MM_{km}$ for some $k,m\in[n]$ with $k\neq m$.
\end{lem}
\begin{proof}
Let $n=2$. Then, by Item \ref{item:three} of Lemma \ref{lem:n_2}, using Notation \ref{nota:elementary}, we find all minimal pairs $(A,B)$: these are $(Z,I)$, $(I,Z)$, $(U_{12}, U_{21})$, $(U_{21}, U_{12})$. Note that $A \neq B$ for each of these pairs. By Notation \ref{nota:Vs}, $U_{12}$ is $V(1;2)$--generic, $U_{21}$ is $V(2;1)$--generic, $I$ is $W(2;1)$--generic, and $Z$ is $W(1;2)\cap Z(2;1)\cap Z(1;2)$--generic. Hence, by Notation \ref{nota:MM_km}, if $(A,B)$ is a minimal pair, then $(A,B)\in \MM_{12}$ or $(A,B)\in \MM_{21}$.

Now let $n \ge 7$. If $(A,B)$ is a minimal pair, then by Lemma~\ref{lem:0_gift_rows}  there exists $p\in[n]$ with $\Sigma(p) = 2$. Hence by Lemma~\ref{lem:1_gift_row}  there exists $p'\in[n]$ with $p\ne p'$ and $\Sigma(p') = 2$. Therefore, Lemma~\ref{lem:2_gift_rows} is applicable which guarantee  that  $A \neq B$ and $(A,B) \in \MM_{km}$, for some $k,m\in[n]$ with $k\neq m$. The proof is complete.
\end{proof}

\begin{cor}\label{cor:Theta}
If  $n \ge 2$, $n \neq 4$, then $\Theta_n=4n - 6$.
If $n = 4$, then $\Theta_n=8$.
\end{cor}
\begin{proof}
If $n=2$ or $n \ge 7$, then the statement  follows from Lemma \ref{lem:necessity} and Item~\ref{item:MM_4n-6} of Corollary~\ref{cor:MM_km}.

Let $n=3$ and let $(A,B)$ be minimal. By Lemma~\ref{lem:two_zeros}, $\Sigma(i) \ge 2$ for all $i\in[n]$.
Then, using Corollary \ref{cor:Theta_estimation}, we get $6=2n \le \Sigma(A,B) = \Theta_3 \le 4n-6=6$, hence $\Theta_3 = 4n-6$.

Let $n=4$ and let $(A,B)$ be minimal. By Lemma~\ref{lem:two_zeros}, $\Sigma(i) \ge 2$ for all $i\in[n]$. In addition, by Example \ref{ex:four}, $\Theta_4 \le \Sigma(A_4,B_4) = 8$.
Then, $8=2n \le \Sigma(A,B) = \Theta_4 \le 8$, hence $\Theta_4 = 8$.

Let $n = 5$ and let $(A,B)$ be minimal. Suppose that $\Theta_5 = \Sigma(A,B) < 4n-6 = 14$.
Then by Lemma~\ref{lem:two_zeros}, $\Sigma(i) \ge 2$ for all $i\in[n]$.
Since $3(n-1) + 2\cdot1=14 > \Theta_5$, there exist at least two rows $p,p'\in[n]$ with $\Sigma(p) = \Sigma(p') = 2$.
Also there is no cost row in $C$ (indeed:  if  a cost row exists, then, by Item~\ref{item:2_duplicates} of Remark~\ref{rem:exclusive} and Lemma~\ref{lem:two_zeros}, we get $\Sigma(A,B) \ge 2(n-2) + 2(n-1) = 4n-6 > \Theta_5$, a contradiction). Hence $p$ and $p'$ are gift rows, by the proof of Lemma \ref{lem:n-3_gift_zeros}. Then, by the proof of Lemma \ref{lem:2_rows_km}, using $5n-11=6n-16=14 > \Theta_5$, we get that there exist $k,m\in[n]\setminus\{p,p'\}$ with $k\neq m$ such that $c_{qt}=\phi_{qt}^{km}$, for all $t \in[n]\setminus \{q,k,m\}$  and $q = p,p'$.
Then, by Items \ref{item:2_km_k}, \ref{item:2_km_m} of Lemma~\ref{lem:2_km_zeros} for row $p$,
and, by Lemma~\ref{lem:two_zeros}, we get $\Sigma(A,B) \ge 2(n-1)+2(n-2)=4n-6$, which is a contradiction with $\Sigma(A,B) < 4n-6$. Hence, $\Theta_5 \ge 4n - 6$, and Corollary \ref{cor:Theta_estimation} completes the proof.

Let $n = 6$ and let $(A,B)$ be minimal. Suppose that $\Theta_6 = \Sigma(A,B) < 4n-6 = 18$. By Lemma~\ref{lem:two_zeros}, $\Sigma(i) \ge 2$ for all $i\in[n]$. Since $3n = 18 > \Theta_6$, there exists at least one row $p\in[n]$ with $\Sigma(p) = 2$. Then, by Lemmas~\ref{lem:2_gift_rows}, \ref{lem:1_gift_row}, we get $A \neq B$ and $(A,B) \in \MM_{km}$, for some $k,m\in[n]$ with $k\neq m$. Hence $\Sigma(A,B) = 4n - 6$, by Item~\ref{item:MM_4n-6} of Corollary~\ref{cor:MM_km}, which is a contradiction with $\Sigma(A,B) < 4n-6$. Hence, $\Theta_6 \ge 4n - 6$, and Corollary \ref{cor:Theta_estimation} completes the proof.
\end{proof}

\begin{rem} Comparing $\Theta_n=\propa(C)=4n-6$ and $\gift(C)=(n-2)(n-3)$, we notice that $\propa(C)\le \gift(C)$ if and
only if $n\ge8$, the case $n=7$ giving $4n-6=22>20=(n-2)(n-3)$. Asymptotically, the ratio  $\gift(C)/\propa(C)$ is  $n/4$.
\end{rem}

\begin{thm}\label{thm:Theta}
Let $n=2$ or $n \ge 7$. Then the pair $(A,B)$ is minimal if and only if $A \neq B$ and $(A,B)\in \MM_{km}$ for some $k,m\in[n]$ with $k\neq m$.
\end{thm}
\begin{proof}
The necessity follows from Lemma \ref{lem:necessity}. Let us prove the sufficiency.

Assume, $n=2$. Then, by Item \ref{item:three} of Lemma \ref{lem:n_2}, using Notation \ref{nota:elementary}, we find all minimal pairs $(A,B)$: $(Z,I)$, $(I,Z)$, $(U_{12}, U_{21})$, $(U_{21}, U_{12})$. Using Notation \ref{nota:Vs}, we get that for $k \neq m$, the sets of $V(k;m)$--generic matrices and of $W(m;k)\cap Z(k;m)$--generic matrices are both equal to $\{U_{km}\}$, the set of $W(k;m)$--generic matrices equals $\{I\}$, and the set of $W(k;m)\cap Z(m;k)\cap Z(k;m)$--generic matrices equals $\{Z\}$. Hence, by Notation \ref{nota:MM_km}, if $A \neq B$ and $(A,B)\in \MM_{km}$ for some $k,m\in[2]$ with $k\neq m$, then $(A,B)$ is minimal.

Now let $n \ge 7$. Then $\Theta_n=4n - 6$, by Corollary \ref{cor:Theta}. If $A \neq B$ and $(A,B) \in \MM_{km}$, for some $k,m\in[n]$ with $k\neq m$, then, by Item~\ref{item:MM_4n-6} of Corollary~\ref{cor:MM_km}, we get $\Sigma(A,B) = 4n - 6 = \Theta_n$, hence $(A,B)$ is minimal.
\end{proof}

The following example shows that Theorem~\ref{thm:Theta} does not hold for $n=3,4,5,6$. It also shows that  few gift rows or no gift  rows is possible for $n\le 6$.
\begin{ex} \label{ex:four}
The following orthogonal pairs $(A_n,B_n)$, $n=3,4,5,6$, are minimal, but $(A_n,B_n)\notin \MM_{km}$ for all $k,m\in[n]$ with $k\neq m$. The minimality of the pairs follows from Corollary \ref{cor:Theta}.
$$A_3=\begin{bmatrix}
0 & - & -\\
- & 0 & -\\
- & - & 0
\end{bmatrix}, \
B_3=\begin{bmatrix}
0 & 0 & 0\\
0 & 0 & 0\\
0 & 0 & 0
\end{bmatrix}, \
C_3=\begin{bmatrix}
{\color{red}0} & {\color{blue}0} & {\color{blue}0}\\
{\color{blue}0} & {\color{red}0} & {\color{blue}0}\\
{\color{blue}0} & {\color{blue}0} & {\color{red}0}
\end{bmatrix}.$$
$$A_4=\begin{bmatrix}
0 & - & - & 0\\
- & 0 & 0 & -\\
- & 0 & 0 & -\\
0 & - & - & 0
\end{bmatrix}, \
B_4=\begin{bmatrix}
0 & - & 0 & -\\
- & 0 & - & 0\\
0 & - & 0 & -\\
- & 0 & - & 0
\end{bmatrix}, \
C_4=\begin{bmatrix}
{\color{red}0}& \phi_{12}^{43} & {\color{blue}0}& {\color{blue}0}\\
\phi_{21}^{34} & {\color{red}0}& {\color{blue}0}& {\color{blue}0}\\
{\color{blue}0}& {\color{blue}0}& {\color{red}0}&\phi_{34}^{21}\\
{\color{blue}0}& {\color{blue}0}& \phi_{43}^{12}&{\color{red}0}
\end{bmatrix}.$$

		$$A_5=\begin{bmatrix}
		0 & - & - & 0 & -\\
		- & 0 & 0 & - & 0\\
		- & 0 & 0 & - & 0\\
		0 & - & - & 0 & -\\
		- & 0 & 0 & - & 0
		\end{bmatrix}, \
		B_5=\begin{bmatrix}
		0 & - & 0 & - & -\\
		- & 0 & - & 0 & -\\
		0 & - & 0 & - & -\\
		- & 0 & - & 0 & 0\\
		- & - & - & 0 & 0
		\end{bmatrix}, \
		C_5=\begin{bmatrix}
		{\color{red}0}& \phi_{12}^{43} & {\color{blue}0}& {\color{blue}0} & \phi_{15}^{43}\\
		\phi_{21}^{34} & {\color{red}0}& {\color{blue}0}& {\color{blue}0} & {\color{blue}0}\\
		{\color{blue}0}& {\color{blue}0}& {\color{red}0}&\phi_{34}^{21} & {\color{blue}0}\\
		{\color{blue}0}& {\color{blue}0}& \phi_{43}^{12}&{\color{red}0} & {\color{blue}0}\\
		\phi_{51}^{34} & {\color{blue}0} & {\color{blue}0} & {\color{blue}0} & {\color{red}0}
		\end{bmatrix}.$$

		$$A_6=\begin{bmatrix}
		0 & - & - & 0 & - & -\\
		- & 0 & - & - & 0 & -\\
		- & - & 0 & - & - & 0\\
		0 & - & - & 0 & - & -\\
		- & 0 & - & - & 0 & -\\
		- & - & 0 & - & - & 0
		\end{bmatrix},\
		B_6=\begin{bmatrix}
		0 & - & - & - & 0 & 0\\
		- & 0 & - & 0 & - & 0\\
		- & - & 0 & 0 & 0 & -\\
		- & 0 & 0 & 0 & - & -\\
		0 & - & 0 & - & 0 & -\\
		0 & 0 & - & - & - & 0
		\end{bmatrix},\
		C_6= \begin{bmatrix}
		{\color{red}0}   & \phi_{12}^{45}   & \phi_{13}^{46}   & {\color{blue}0}  & {\color{blue}0}& {\color{blue}0}\\
		\phi_{21}^{54}   & {\color{red}0}   & \phi_{23}^{56}   & {\color{blue}0}& {\color{blue}0}  & {\color{blue}0}\\
		\phi_{31}^{64}   & \phi_{32}^{65}   & {\color{red}0}   & {\color{blue}0}& {\color{blue}0}& {\color{blue}0}\\
		{\color{blue}0}  & {\color{blue}0}& {\color{blue}0}& {\color{red}0}   & \phi_{45}^{12}   & \phi_{46}^{13}\\
		{\color{blue}0}& {\color{blue}0}  & {\color{blue}0}& \phi_{54}^{21}   & {\color{red}0}   & \phi_{56}^{23}\\
		{\color{blue}0}& {\color{blue}0}& {\color{blue}0}  & \phi_{64}^{31}   & \phi_{65}^{32}   & {\color{red}0}
		\end{bmatrix}.$$
\end{ex}

\subsection{Self--orthogonal matrices}\label{subsec:self_ortho}
If $A^2=Z_n$, then  $A$ is called \emph{self--orthogonal.}
Now we  set $A=B$   and let $C$ be the indicator matrix of the pair $(A,A)$.
Then, $a_{sk}=a_{kt}=0$ with $s,t,k\in[n]$ pairwise different and $a_{st}\neq0$ yield a cost zero $c_{st}=\phi_{st}^{kk}$.

Notice that $\Theta_n\le 2 \Theta_n^{\Delta}$ where $\Theta_n^{\Delta}$ is the minimum over the diagonal $\Delta$ of $M_n^N\times M_n^N$
\begin{equation}
\Theta_n^{\Delta}:=\min_{A\in M_n^N}\{\nu(A)-n : A^2=Z_n\}.
  \end{equation}
By Corollary \ref{cor:self_orth} and Lemma \ref{lem:nu}, we
know that  $\Theta_n^{\Delta} \le 2n - 2$.
The aim of this subsection is to prove that $\Theta_n^{\Delta}=2n-2$ and so $\Theta_n= 2\Theta_n^{\Delta}-2$ (this agrees with the
minimal values of $\nu(A)-n$ found  in Lemma \ref{lem:nu}).
The following theorem is an
analogue of Theorem \ref{thm:Theta} for the case $A = B$. Observe that gift zeros do not exist, as remarked
in Item~\ref{item:gift_dont_exist} of Remark~\ref{rem:exclusive}, and  we cannot use Lemmas \ref{lem:MM_km},
\ref{lem:2_km_zeros}, \ref{lem:2_rows_km},  \ref{lem:n-3_gift_zeros} and  \ref{lem:0_gift_rows}.
\begin{thm}\label{thm:self_orthogonal_minimal}
Let $n\ge 5$.
The matrix $A\in \norma n$ is self--orthogonal  with the minimal number of off--diagonal zeros $\Theta_n^{\Delta}$ if and only
if there exists $k\in [n]$ such that $A$ is $V(k;k)$--generic.
\end{thm}
\begin{proof}
Let $A$ be self--orthogonal with minimal number of off--diagonal zeros, i.e.,
 $\nu(A) - n = \Theta_n^\Delta$.
By Lemma \ref{lem:two_zeros}, $\nur A i - 1 \ge 1$ for all $i \in [n]$.
Show that there exist at least two rows $s,s'$ with $\nur A s - 1 = 1$ and $\nur A {s'} - 1 = 1$ (indeed, if $\nur A i - 1 \ge 2$
for $n-1$ rows, then $\nu(A) - n \ge 2(n-1) + 1 = 2n - 1 > 2n - 2$, which contradicts with $\Theta_n^\Delta=\nu(A) - n \le 2n - 2$).
Let $a_{sk}=0$ for some $k \in [n] \setminus \{s\}$ and $a_{s'k'}=0$ for some $k' \in [n] \setminus \{s'\}$.
Since $\ro C s$ contains only one propagation zero, $C$ does not contain gift zeros, and $\ro C s$ is
zero then, by Items~\ref{item:exclusive} and \ref{item:cost_row} of Remark~\ref{rem:exclusive},
$\ro C s$ is $k$--cost and $c_{st}=\phi_{st}^{kk}$, for all $t \in [n] \setminus \{s,k\}$.
Similarly, $\ro C {s'}$ is $k'$--cost and $c_{s't}=\phi_{s't}^{k'k'}$, for all $t \in [n] \setminus \{s',k'\}$.
By Item \ref{item:cost_def} of Definition \ref{dfn:gift},  $a_{kt}=0$ for all $t \in [n] \setminus \{s,k\}$, and $a_{k't}=0$ for all
$t \in [n] \setminus \{s',k'\}$.
If $k \neq k'$, then $\nur A k - 1 \ge n-2$ and $\nur A {k'} - 1 \ge n-2$, whence using $\nur A i - 1 \ge 1$ for
$i \in [n] \setminus \{k,k'\}$ we get $\Theta_n^\Delta=\nu(A) - n \ge 2(n-2) + (n-2) = 3n - 6 > 2n - 2$, a contradiction. Thus,
$k=k'$ and $a_{kt}=0$ for all $t \in [n] \setminus \{k\}$.
Hence $\nur A k - 1 \ge n-1$.
Then $\nu(A) - n \ge (n-1) + (n-1) = 2n - 2$, hence $\nu(A) - n = 2n - 2$.
Then $\nur A k - 1 = n-1$ and $\nur A i - 1 = 1$ for all $i \in [n] \setminus \{k\}$.
Since for all $i \in [n] \setminus \{k\}$ $\ro C i$ contains only one propagation zero, similarly we get that $\ro C i$ is $k$--cost
and $c_{it}=\phi_{it}^{kk}$, for all $i \in [n] \setminus \{k\}$ and for all $t \in [n] \setminus \{i,k\}$, which completes the proof.

For the sufficiency, let the matrix $A$ be $V(k;k)$--generic. Then, by Corollary~\ref{cor:self_orth} and Lemma~\ref{lem:nu} $A$
is self--orthogonal and $\nu(A) - n = 2n - 2$. Using the proved necessity we get the desired result.
\end{proof}

The following example shows that  Theorem \ref{thm:self_orthogonal_minimal} is not true for $n=3$.
\begin{ex}
$A=\left[\begin{array}{rrr}  0&-1&0\\0&0&-1 \\-1&0&0\end{array}\right]$ is self--orthogonal,   by Example  \ref{ex:easy_1}, but
$A\not\in V(k;k)$,  for $k\in [3]$.
\end{ex}

\section{Orthogonality by bordering} \label{sec:bordering}
In this section we  study what happens with orthogonality after adding a row and a column, which  enables us to construct orthogonal
pairs of arbitrary sizes. We assume $n\ge2$.

Let the matrix  $A=
\left[\begin{array}{ccc}
B&v\\
w^T&0\\
\end{array} \right]\in\norma n$ be decomposed into blocks, with $B\in \norma {n-1}$,  and  $v,w$ non--positive vectors.

 \begin{prop}[Orthogonality by bordering]\label{lem:border}
Let $A_k=
\left[\begin{array}{ccc}
B_k&v_k\\
w_k^T&0\\
\end{array} \right]\in\norma n$  be as above, with $k=1,2$. If $B_1B_2=Z_{n-1}=B_2B_1$, then  $A_1A_2=Z_n=A_2A_1$  if and only if
$B_1v_2\oplus v_1=B_2v_1\oplus v_2$ and $w_1^T B_2\oplus w_2^T=w_2^T B_1\oplus w_1^T$ are zero vectors.
\end{prop}
\begin{proof} Easy computations show that
$$A_1A_2=
\left[\begin{array}{ccc}
Z_{n-1}& B_1v_2\oplus v_1\\
w_1^T B_2\oplus w_2^T&0\\
\end{array} \right],\
A_2A_1=
\left[\begin{array}{ccc}
Z_{n-1}& B_2v_1\oplus v_2\\
w_2^T B_1\oplus w_1^T&0\\
\end{array} \right].$$
The rest is immediate.
\end{proof}

\begin{nota}
 For $i,j\in[n]$ let $P^{ij}=[p_{kl}]$ be the \emph{permutation matrix} corresponding to the transposition $(ij)$, i.e.,  $p_{kl}= 0$ if $(k,l)=(i,j)$ or $(j,i)$ or $k=l\in[n]\setminus\{i,j\}$,
$p_{kl}=-1$, otherwise.\footnote{Properties
of $P^{ij}$ are analogous in classical and tropical linear algebra.
In general, $P^{ij}$ is not a normal matrix.}
\end{nota}

\begin{dfn}[Orthogonal set of  a matrix]\label{dfn:Or} For any subset $S\subseteq  M_n^N$, define the set
 $\Or(A)_S:=\{B\in S: AB=Z_n=BA\}$. Write $\Or(A)$ if $S$ is  the ambient                                                                  space.
\end{dfn}

\begin{lem}[Decreasing size] \label{lem:decreasing_size}
If$A \in\norma n$ and there exists $i\in[n]$ such that both the $i$--th row and the $i$--th column of $A$ have no zero entries except
on the main diagonal, then
$P^{ni}AP^{ni}=\left[\begin{array}{ccc}
B&v\\
w^T&0\\
\end{array} \right]$, with $v, w\in \R^{n-1}_{\le0}$ without zero entries, and
$$\Or(P^{ni}AP^{ni})=\left\{
\left[\begin{array}{ccc} D&Z_{(n-1) \times 1}\\ Z_{1 \times (n-1)}&0\\
\end{array} \right]:\
D \in \Or(B)
\right\}.$$
In particular, $|\Or(A)|=|\Or(B)|$. \qed
\end{lem}

 \begin{proof}
 First, it is easy to check that $P^{ni}\Or(A)P^{ni}=\Or(P^{ni}AP^{ni})$. Now, set  $A_1=P^{ni}AP^{ni}$, $B_1=B$, $v_1=v$, $w_1=w$ in
 Lemma \ref{lem:border} and notice that, by hypothesis, the vectors $v_1$ and $w_1$ never vanish. If $A_1A_2=Z_n=A_2A_1$, then  the vector
 $B_1v_2\oplus v_1=B_1v_2$ is  zero, by  Lemma \ref{lem:border}. From normality  of $B_1$ and  monotonicity of $\odot$,  it follows that
 $I_{n-1}\le B_1\le Z_{n-1}$ and  $v_2\le B_1v_2\le Z_{n-1}v_2$. But the vector $Z_{n-1}v_2$ is constant, so it vanishes if and only
 if $v_2$  vanishes. Similarly, $w_2$ vanishes. Then the last row and column in $A_2$  are zero and, if we write $D=B_2$, we get the statement.
 \end{proof}

\begin{cor} \label{cor:strictly}
The only matrix orthogonal to a strictly normal matrix of order $n$ is $Z_n$.
\end{cor}
\begin{proof}
Let $A\in M_n^{SN}$ and,  for $j\in [n]$, let $A_j$ denote the principal submatrix of $A$ of  order $j$. Then applying Lemma \ref{lem:decreasing_size} repeatedly for $j\in[n]$,  $i=j$, $A=A_j$ and $B=A_{j-1}$, we obtain that
$|\Or(A_n)|=|\Or(A_{n-1})|=\cdots=|\Or(A_2)|=1$.
Since $Z\in \Or(X) $ for any $X\in \norma n$, the result follows. \end{proof}

\begin{cor}[Self--orthogonality by bordering]\label{cor:self_ortho}
Let $A=
\left[\begin{array}{ccc}
B&v\\
w^T&0\\
\end{array} \right]\in\norma n$ be  decomposed into blocks, with $B\in \norma {n-1}$. If $B$ is self--orthogonal, then $A$ is self--orthogonal if and only if $Bv$ and $w^TB$ are zero vectors.
\end{cor}
\begin{proof}
We have $A^2=\left[\begin{array}{ccc}
B^2\oplus vw^T&Bv\oplus v\\
w^TB\oplus w^T&0\\
\end{array} \right]$, where the matrix $vw^T$ is non--positive, but not necessarily  normal.
By hypothesis, we have $B^2=Z$, whence $ B^2\oplus vw^T= Z\oplus vw^T=Z$ and   $Bv\oplus v=(B\oplus I)v=Bv$.
\end{proof}

\section{ Three  orthogonality graphs}\label{sec:graphs}

In this section  we  compute the diameter and girth  of three  types of
graphs related to the
orthogonality relation.  Graphs can have loops  but no  multiple edges. We assume $n\ge3$ since otherwise the graphs under
consideration are disconnected and more or less trivial. In the first and most intuitive graph,  denoted $\ORTHO$, vertices are matrices and an edge
between two matrices $A,B$ means that $A,B$ are mutually orthogonal. A loop stands for a self--orthogonal matrix.

\medskip
Let $\Gamma=(V,E)$ be a graph with the vertex set $V$ and edges $E\subseteq V\times V$. We consider three different sets of vertices.

\begin{dfn}
A {\em path} (or {\em walk)\/} is a sequence  $v_0, e_1, v_1, e_2, v_2, \ldots , e_k, v_k$ of vertices $v_0,\ldots, v_k\in V$ and edges $e_1,\ldots, e_k\in E$
where $e_i=(v_{i-1},v_i)$ for all $i=1,\ldots,k$.
If $v_0=v_k$, then the path is {\em closed\/}.
The {\em length of  the former path\/} is $k$.
 A path is {\em elementary\/} if all the edges are distinct.
 A {\em cycle\/} is a closed elementary path.
\end{dfn}
\begin{dfn}
The {\em girth\/} of a graph $\Gamma$ is the length of the shortest cycle in $\Gamma$ which is not a loop.
\end{dfn}
\begin{dfn} The graph $\Gamma$ is said to be {\em connected\/}
if it is possible to establish a path from any vertex to any other vertex of $\Gamma$.
\end{dfn}
\begin{dfn}
The {\em distance $\dist(u,v)$\/} between two vertices~$u$ and~$v$ in a~graph $\Gamma$ is the length of the shortest path between them.
If~$u$ and~$v$ are unreachable from each other, then we set $\dist(u,v)=\infty$. It is assumed that $\dist(u,u)=0$ for any vertex $u$.
\end{dfn}
\begin{dfn}
The {\em diameter $\diam(\Gamma)$\/} of a graph $\Gamma$ is the maximum of distances between vertices, for all pairs of vertices in $\Gamma$.
\end{dfn}

Recall that the  absorbing property (\ref{eqn:absorbing})  in Remark \ref{rem:absorbing} says that $Z$ is orthogonal to every
matrix. On the other hand, if $A$ is
strictly normal, then the only
matrix orthogonal to $A$ is $Z$, by Corollary  \ref{cor:strictly}.
So, it is reasonable to consider
$$(\norma n)^*:=\norma n\setminus\left( \snorma n\cup\{Z\}\right)$$
as a set of vertices. Namely, we delete the vertex which is connected with all other vertices  as well as the set
of isolated vertices.

Recall Notations \ref{nota:Vs}. In view of Corollary \ref{cor:suff_conditions}, other interesting sets of matrices
 are $$VNL:=\bigcup_{p,q\in [n],\, p \neq q}V(p;q),\qquad VNL^*=VNL\setminus \{Z\},$$
$$WNL:=\bigcup_{p,q\in [n],\, p \neq q}W(p;q),\qquad WNL^*=WNL\setminus \{Z\}.$$

Now we define the corresponding graphs.

\begin{dfn}\label{dfn:ORTHO}  The vertex set of the graph $\ORTHO$ is $(\norma n)^*$.
 Matrices $A, B\in (\norma n)^*$ are joined by an edge in  $\ORTHO$ if and only if
$AB=Z=BA$. In particular, loops in $\ORTHO$  correspond to self--orthogonal matrices.
\end{dfn}

Lemma \ref{lem:starting_point} is  motivation for the following Definition.
\begin{dfn}\label{dfn:VNL}The vertex set of the graph $\VNL$ is $VNL^*$.
Matrices $A, B\in VNL^*$ are  joined by an edge in  $\VNL$ if and only if  there exist $p,q\in[n]$
with $A\in V(p;q)$ and $B\in V(q;p)$.
\end{dfn}

\begin{dfn}\label{dfn:WNL}The vertex set of the graph $\WNL$ is $WNL^*$.
Matrices $A, B\in WNL^*$ are  joined by an edge in  $\WNL$ if and only if  there exist $k,m\in[n]$
with $A$ and $B$ satisfying one of the conditions $1, 2, 3$ in Corollary \ref{cor:suff_conditions}.
\end{dfn}

\begin{rem}
Notice that $VNL^*, WNL^* \subseteq (\norma n)^*$ and, by Corollary \ref{cor:suff_conditions}, the graphs $\VNL, \WNL$ are subgraphs of  $\ORTHO$.
\end{rem}

\begin{prop} \label{prop:girth}
	Let $n \ge 3$. Then $\girth(\ORTHO)=\girth(\VNL)=3$.
\end{prop}
\begin{proof}
Recall Notations \ref{nota:Vs} and take $A_1\in V(1,2;3)$, $A_2\in V(2,3;1)$ and  $A_3\in V(3,1;2)$. Then $A_1-A_2-A_3-A_1$ make a cycle in $\VNL$,
so $\girth(\ORTHO)\le \girth(\VNL)\le3$. Since the  graphs under consideration have no multiple edges, the proof is complete.
\end{proof}

\begin{prop} \label{prop:diam}
Let $n \ge 3$. Then $\ORTHO$ is connected and $\diam(\ORTHO)=3$.
\end{prop}
\begin{proof}
Let $A, B$ be matrices corresponding to
distinct vertices in the graph $\ORTHO$.
Each of them has at least one off--diagonal zero. Assume   $a_{ij}=0=b_{km}$, with $i\neq j$ and $k\neq m$.  We can assume that $(i,j)\neq (k,m)$.
Then, by Notations \ref{nota:elementary} and Corollary \ref{cor:propagation}, $A \in \Or(E_{ij})_{(\norma n)^*}$, $B \in \Or(E_{km})_{(\norma n)^*}$, and also
$E_{km} \in \Or(E_{ij})_{(\norma n)^*}$, since $(i,j)\neq (k,m)$.
So the path $A-(E_{ij})-(E_{km})-B$ shows that $\diam(\ORTHO) \le 3$.

Now, consider the matrices $U_{ij}, U_{km}$, with $i,j\in [n]$, $i\neq j$,  $k,m\in [n]$, $k\neq m$ and $(i,j)\neq (k,m)$.
Since $n \ge 3$, there exists $l \in [n] \setminus \{i, k\}$, and   the $l$--th row of the product  $U_{ij} U_{km}$ is non--zero. Hence,
$U_{ij}$ and $U_{km}$ are not orthogonal and
$\dist(U_{ij}, U_{km}) > 2$, which completes the proof.
\end{proof}

\begin{prop} \label{prop:diam_nl}
Let $n \ge 3$. Then $\diam(\VNL)=2$.
\end{prop}
\begin{proof}
Let $A, B$ be matrices corresponding to two distinct
vertices in the graph $\VNL$.
If   for some $p,q\in [n]$  both $A, B \in V(p;q)$ then we have a path $A - C - B$, where $C \in V(q;p)$ is non--zero.
Now assume $A \in V(p;q), p \neq q$, $B \in V(c;d), c \neq d$, and $p \neq c$.
If $|\{p, q, c, d\}| = 2$, then $(p, q) = (d, c)$ and we have a path $A - B$.
Otherwise the set $\{p, q, c, d\}$ contains at least $3$ different numbers.
Consider the set $S = V(q;p) \cap V(d;c)$ and a $S$--generic matrix $C$. We show  that $C \neq Z$.
Indeed, in $C$ the $q$--th and the $d$--th rows are zero, and also the $p$--th and the $c$--th columns are zero.
The number of zeros in $C$ satisfies   $\nu(C) \le 2n + 2(n - 2) + n = 5n-4$.
But since $|\{p, q, c, d\}| \ge 3$, then at least $3$ diagonal zeros (among $c_{qq}, c_{pp}, c_{dd}, c_{cc}$)
intersect with zero rows and columns of $C$. So $\nu(C) \le 5n - 4 - 3 = 5n - 7$. It is clear  that $5n - 7 < n^2$, so $C$ is not the zero matrix.
So we have a path $A-C-B$, by Definition~\ref{dfn:VNL}. We have shown that $\diam(\VNL) \le 2$.

Now, consider a  $V(i;j)$--generic matrix $L$, $i,j\in [n]$, $i\neq j$, and a  $V(k;m)$--generic matrix $M$, $k,m\in [n]$, $k\neq m$, $(i,j)\neq (k,m)$, $(i,j)\neq (m,k)$. Then, by Definition~\ref{dfn:VNL}, $L$ and $M$ are not joined by an edge in $\VNL$ and
$\dist(L, M) > 1$, which completes the proof.
\end{proof}

\begin{prop} \label{prop:diam_wnl}
If $n \ge 4$, then $\diam(\WNL)=2$.
\end{prop}
\begin{proof}
Let $A, B$ be matrices corresponding to two distinct arbitrary vertices in the graph $\WNL$.
If,  for some $p,q\in [n]$   both $A, B \in W(p;q)$ then we have a path $A - C - B$, where $C \in W(q;p) \cap Z(q;p) \cap Z(p;q)$ is non--zero.
Assume $A \in W(p;q), p \neq q$, $B \in W(c;d), c \neq d$, and $p \neq c$.
If $|\{p, q, c, d\}| = 2$, then $(p, q) = (d, c)$ and we have a path $A - C - B$, where $C \in V(p;q) \cap V(q;p)$ is non--zero.
Otherwise the set $\{p, q, c, d\}$ contains at least $3$ different numbers.
Consider the set $S = V(q;p) \cap Z(p;q) \cap V(d;c) \cap Z(c;d)$ and a $S$--generic matrix $C$. We show  that $C \neq Z$.
Indeed, in $C$ the $q$--th and the $d$--th rows are zero, and also the $p$--th and the $c$--th columns are zero, besides $c_{pq} = c_{cd} = 0$.
The number of zeros in $C$ satisfies   $\nu(C) \le 2n + 2(n - 2) + 2 + n= 5n-2$.
But since $|\{p, q, c, d\}| \ge 3$, then at least $3$ diagonal zeros (among $c_{qq}, c_{pp}, c_{dd}, c_{cc}$)
  have been counted twice. So $\nu(C) \le 5n - 2 - 3 = 5n - 5$.  We have $5n - 5 < n^2$, so $C$ is not   the  zero matrix.
So we have a path $A-C-B$, by Definition~\ref{dfn:WNL}. We have shown that $\diam(\WNL) \le 2$.

Now, consider a $W(i;j)$--generic matrix $L$, with $i,j\in [n]$, $i\neq j$, and a $W(k;m)$--generic matrix $M$, with  $k,m\in [n]$, $k\neq~m$, $(i,j)\neq (k,m)$. Then, by Definition~\ref{dfn:WNL}, $L$ and $M$ are not joined by an edge in $\WNL$ and
$\dist(L, M) >~1$, which completes the proof.
\end{proof}

The following example shows that the above Lemma is not true if $n=3$.

\begin{ex}
For
$A = \begin{bmatrix}
0 & -1 & 0 \\
-1 & 0 & -1 \\
-1 & 0 & 0
\end{bmatrix} \in W(1;2),\ B = \begin{bmatrix}
0 & 0 & -1 \\
-1 & 0 & 0 \\
-1 & -1 & 0
\end{bmatrix} \in W(1;3)$
we get that $\dist(A,B) =3$ in $\WNL$.
\end{ex}
Indeed,
the following path shows that $\dist(A,B) \le 3$ in $\WNL$ $$A - \begin{bmatrix}
0 & 0 & -1 \\
0 & 0 & 0 \\
0 & -1 & 0
\end{bmatrix} - \begin{bmatrix}
0 & -1 & 0 \\
0 & 0 & -1 \\
0 & 0 & 0
\end{bmatrix} - B.$$

Note that matrices $A$ and $B$ are not mutually orthogonal. Let us show that $\dist(A,B) > 2$ in $\WNL$.
Indeed, suppose that we have a path $A - D - B$ in $\WNL$. By definition, $A$ and $D$ are connected if for some $k, m$ one of the following items holds:
\begin{enumerate}
\item $A\in W(k;m)\cap Z(k;m)\cap Z(m;k)$ and $D\in W(m;k)$,
\item $A\in W(k;m)\cap Z(m;k)$ and $D\in W(m;k)\cap Z(k;m)$,
\item $A\in W(k;m)$ and $D\in W(m;k)\cap Z(k;m)\cap Z(m;k)$. \label{item:last}
\end{enumerate}
The structure of $A$ implies that the only case left is the item~\ref{item:last} with $k = 1, m = 2$.
Similarly, for $B$ with $k = 1, m = 3$. But $D\in W(2;1)\cap Z(1;2)\cap Z(2;1)$ and $D\in W(3;1)\cap Z(1;3)\cap Z(3;1)$ only if $D$ is zero and the proof is complete.

\section*{Acknowledgments}

The work of the first and the second authors is financially supported by the grant RSF 17--11--01124.
The work of the third author is partially supported by  Ministerio de Econom\'{\i}a y Competitividad, Proyecto I+D MTM 2016--76808--P and
910444 UCM group.

The authors are thankful to the referees for useful comments which helped to improve the paper.

\end{document}